\documentclass[12pt,a4paper]{amsart}
\usepackage[letterpaper, margin=0.8in]{geometry}
\usepackage{amsfonts,amssymb,amsthm,epsf, graphics, bbm, mathabx, verbatim, caption, subcaption, enumerate}
\usepackage{comment}
\usepackage{hyperref}
\usepackage{diagbox}
\usepackage{amsmath}
\usepackage{amsthm}
\usepackage{amssymb}
\usepackage{cite}
\usepackage{tikz}
\usetikzlibrary{patterns}
\usepackage{float}
\usepackage{booktabs,makecell}

\newtheorem{proposition}{Proposition}

\newtheorem{lemma}{Lemma}
\theoremstyle{definition}
\newtheorem{definition}{Definition}
\newtheorem{theorem}{Theorem}
\newtheorem{corollary}[theorem]{Corollary}
\newtheorem*{theoremA}{Theorem A}
\newtheorem*{theoremB}{Theorem B}

\newtheorem{example}{Example}

\theoremstyle{remark}
\newtheorem{remark}{Remark}

\renewcommand{\L}{\mathcal{L}}
\newcommand{\R}{\mathbb{R}}

\newcommand{\p}{\partial}
\newcommand{\eps}{\varepsilon}
\newcommand{\HH}{\mathcal{H}}

\newcommand{\M}{\mathcal{M}}

\newcommand{\norm}[1]{\left\|#1\right\|}

\title{A Note on Maximal Operators for Moment Curves}
\author{Chenjian Wang}
\date{}

\begin{document}

\begin{abstract}
We consider a type of maximal operators associated to moment curves in $\R^d,d\geq 3.$ We derive $L^p$ mapping properties for these operators. In a special case, the estimate is sharp.

\end{abstract} 

 \maketitle

\let\thefootnote\relax

\section{Introduction}
Geometric measure theory is the study of the 
 geometric properties of sets, typically in Euclidean space, through measure theory. An interesting genre of problems in this area is determining the size, usually quantified via Lebesgue measure or dimensions, of sets that are quantitatively known to be large. A prototypical example of such a problem is the notorious Kakeya conjecture which aims to quantify the size of a set that contains a unit line segment in every direction. Such sets are called Kakeya sets. Besicovitch \cite{Besicovitch_1919} showed that, contrary to intuition, Kakeya sets can have Lebesgue measure zero. This led to the conjecture which states that Kakeya sets must have full Hausdorff dimension. 
 The problem has been solved for $d=2$ \cite{Davies_1971,Cordoba_1977}. Recently, Wang and Zahl claimed the resolution of three-dimensional case \cite{wang2022stickykakeyasetssticky,AssouadDimKakeya,wang2025volumeestimatesunionsconvex}  and the problem is still open in all higher dimensions \cite{Bourgain1991,Wolff1995,Katz2000}.

 The intense interest in the Kakeya conjecture has led to the study of related problems in the ``curved" setting.
 For $d=2$, Besicovitch and Rado \cite{BR} and Kinney \cite{Kinney} constructed a Lebesgue null set that contains all circles with radius $r\in[1,2]$. Inspired by this, such a set in $\R^d,d\geq 2$ (replace circles by $r \mathbb S^{d-1}$ for $d\geq 3$) is called Besicovitch-Rado-Kinney set, or a BRK set for short. One can ask a similar question to the Kakeya conjecture: what is the Hausdorff dimension of a BRK set? The problem has been solved for all $d\geq 2$: for $d\geq 3$ in \cite{Kolasa1999OnSV} and $d=2$ in \cite{Wolff1997}. It is now known that all BRK sets in $\R^d,\ d\geq 2$ have full Hausdorff dimension.

In this note, we consider problems of a similar flavor in $\R^d,d\geq 3$. Let us denote by 
\begin{equation*}
\gamma(t):=(t,t^2,...,t^d)
\end{equation*}
the standard moment curve in $\R^d$. We want to study sets that contain many affine copies of $\gamma$. More precisely, we introduce the following BRK-type sets. Throughout the note, we denote 
\begin{equation}\label{sands'}
    s'=d+1-s.
\end{equation}

\begin{definition}[$s'$-parameter moment BRK type sets]\label{MomentCurveBRK}
    Assume $d\geq 3$. For $s'\in \{1,2,...,d\}$, a compact set $S_{s'}\subseteq \R^d$ is called a \emph{$s'$-parameter moment BRK type sets} if for all $(x_{s+1},...,x_{d-1},x_d,r)\in [-1/4,1/4]^{s'-1}\times [1/2,2]\subseteq \R^{s'}$, there is $(x_1,...,x_{s})\in \R^{s}$ such that the (truncated) moment curve 
\begin{equation}\label{mcset}
   H(\mathbf x,r):= \{\mathbf x+r\gamma(t):t\in [-1,1]\}\subseteq S_{s'},
\end{equation}
where $\mathbf x=(x_1,...,x_d)$.
\end{definition}

In the following parts of the note,    
\begin{enumerate}[(i)]
    \item we will denote the range of each coordinate of $\mathbf x$ and $r$ by \begin{equation}\label{intervals}
    I_1:=[-1/4,1/4],\ I_2:=[1/2,2].
\end{equation}
and 
$$\overline{\mathbf x}_{s}=(x_1,...,x_s),\quad \underline{\mathbf x}_{s'-1}=(x_{s+1},...,x_d).$$
When $s'$ is fixed, we will abbreviate these as $\overline{\mathbf x}$ and $\underline{\mathbf x}$.
\item Denote the isotropic $\delta$-neighborhood of the moment curve $H(\mathbf x,r)$ as $H_\delta(\mathbf x,r)$,  More precisely,
    \begin{equation*}
        H_\delta(\mathbf x,r):=\bigcup_{\mathbf y\in H(\mathbf x,r)}B(\mathbf y,\delta),
    \end{equation*}
where $\mathbf y=(y_1,...,y_d)\in\R^d$ and $B(\mathbf y,\delta
)$ is the ball with center $\mathbf y$ and radius $\delta$. It is elementary to see that when $r\in [1/2,2]$, for all $\mathbf x\in\R^d,$
\begin{equation*}
    \mathcal L^d(  H_\delta(\mathbf x,r))\approx \delta
    ^{d-1}.
\end{equation*}
\item All dimensions or $\dim(\cdot)$ refer to Hausdorff dimension.
\item As conventions, $A \lesssim B$ or $B\gtrsim A$ represents ``$\exists C>0$ (independent of $\delta$ but may depend on other irrelevant parameter such as $d$), such that $A\leq CB$". If $C$ depends on some irrelevant parameter $\tau$, we may also denote $A\lesssim C_\tau B$ as $A\lesssim_\tau B.$
\end{enumerate}

To state the dimension estimates for the $s'$-parameter moment BRK type sets, we introduce the following definition.

\begin{definition}[$s$-parameter maximal function associated to moment curves]\label{def} Fix $\delta>0$. For a Schwartz function $f:\R^d\to \mathbb C$, $d\geq3,$ define the \textit{$s$-parameter maximal function associated to moment curves} $\M^{s}_\delta f  :I_1^{s'-1}\times I_2\to \mathbb R_+$ as
\begin{equation}\label{maxopt1}
     \M^{s}_\delta f(\underline{\mathbf x}_{s'-1},r):=
         \sup_{\mathbf x\in\R^s}\frac{1}{\mathcal L^d( H_\delta(\mathbf x,r))}\int_{H_\delta(\mathbf x,r)}|f(\mathbf y)|d\mathbf y,s=1,2,...,d
\end{equation}
    where $\mathcal L^d$ is the $d$-dimensional Lebesgue measure and $\ell$ is the arc length measure. 
\end{definition}
Recall our notation in \eqref{sands'}. In the parameter space $\R^{d+1}=\R^s\times \R^{s'}$, $s$ is the dimension or number of parameters we take supremum over. $s'$ is the dimension or number of parameters left, which is also the number of variables of the maximal functions in \eqref{maxopt1}.

By applying Wolff's circular maximal estimate in \cite{Wolff1997} and sharp (exponent) local smoothing estimate by Ko, Lee, and Oh \cite{Ko_Lee_Oh_2023}, the following maximal estimates can be obtained.

\begin{theorem}[maximal estimates] \label{MaxEst}
     Fix $d\geq 3 $ and $ \eps>0$, then for all $ p\geq p_d,$ there exists $ C_\eps>0$ depending on all the aforementioned quantities and $s$, such that
\begin{equation}\label{max1}
    \norm{\M^s_\delta f}_{L^p(I_1^{s'-1}\times I_2)}\leq C_\eps \delta^{-\alpha(s)}\norm{f}_{L^p{(\R^d)}}.
\end{equation}
Here,
\begin{equation*}
    p_d=\begin{cases}
        3,&\text{if }s=d,\\
        4d-2,&\text{if } s=1,...,d-1,
    \end{cases}
\end{equation*}
and
\begin{equation}\label{ExponentOfTheEstimate}\alpha(s)=
    \begin{cases}
        \eps+\frac{s-2}{p},&s=2,...,d,\\
        0,&s=1.
    \end{cases}
\end{equation}
When $s=1,$ $C_\eps$ is an absolute number independent of $\eps$.
\end{theorem}
Note that when $s=d$, the result $p_d= 3$ is better than the general results obtained from local smoothing $p_d\geq 4d-2$. 
\begin{remark}
    When $s=1$, since the operator norm is independent of $\delta,$ by a standard limiting argument, one can deduce 
    \begin{equation*}
        \norm{\mathcal M^1f}_L^p\leq C\norm{f}_p,
    \end{equation*}
    and replace $\mathcal M^1_\delta$ by $\mathcal M^1$, where 
    \begin{equation*}
        \mathcal M^1f:=\sup_{\mathbf x\in\R^d}\int_{H_(\mathbf x,r)}|f(\mathbf y)|d\ell(\mathbf y).
    \end{equation*}
    To unify notations, we still use $\mathcal M^1_\delta.$
\end{remark}

When $s=d$, the estimate \eqref{max1} is sharp in the following sense:
\begin{enumerate}
    \item (sharpness of the exponent of $\delta$) If $p\geq  3$,  $\not\exists \beta<\frac{d-2}{p}$ such that the following holds
    \begin{equation*}
          \norm{\M^d_\delta f}_{L^p([1/2,2])}\leq C_\eps \delta^{-(\eps+\beta)}\norm{f}_{L^p{(\R^d)}}.
    \end{equation*} 
    \item (sharpness of the range of $p$) If $p<3$, then an estimate with the form $s=d$ in \eqref{max1} cannot hold.
\end{enumerate}

By a standard ``maximal estimate to dimension estimate" argument (see Section \ref{Fthm2toThm1}), the following corollary is straightforward.  
Recall the notation in Definition \ref{MomentCurveBRK}.

\begin{corollary}[dimension estimates]\label{dimest}For $d\geq 3$ and $s'\in \{1,2,...,d\}$, 
\begin{equation}\label{DimEstInThm2}
    \dim(S_{s'})\geq \min\{s'+1,d\}.
\end{equation}
\end{corollary}
The dimension estimates \eqref{DimEstInThm2} are sharp in the sense that there exist $S_{s'}$ such that $\dim(S_{s'})=\min\{s'+1,d\}.$ All sharp examples will be presented in Section \ref{examplesforsharpness}. In particular, for the case $s'=d$, $S_d$ has positive Lebesgue measure, see Proposition \ref{MaxToDim}.  

\begin{remark}
     Heuristically, $S_{s'}$ is like an $s'$-dimensional union of $1$-dimensional moment curves. Therefore, it is reasonable to guess the resulting set has Hausdorff dimension $s'+1$. In fact, Corollary \ref{dimest} is special case of the results in Ham, Ko, Lee, and Oh \cite{HKLO}. There they generalized the local smoothing in \cite{Ko_Lee_Oh_2023} to fractal setting (similar argument can be found in \cite{Fractal}). This provides more flexibility in the choice of curve families, as \cite{HKLO} allows $s'$ to take arbitrary values in $(0,d]$ in an appropriate sense.
\end{remark}

\begin{remark}
   The methods applied to obtain maximal estimates for $s=d$ and $s=1,...,d-1$ are different. It may be interesting that for $\mathcal{M}^d_\delta$, the torsion of moment curve is not needed, even the result is sharp under both senses. For $\mathcal{M}^s_\delta,s=1,...,d-1$, the torsion is necessary to obtain the sharp local smoothing in \cite{Ko_Lee_Oh_2023}. 
   
   This is also reflected in the dimension estimates: It is the curvature in $s'$ transversal direction(s) that makes $S_{s'}$ having dimension at least $\min\{s'+1,d\}$. For $\mathcal M_\delta^d$ which corresponds to $\dim(S_1)\geq 1+1,$ only planar curvature is needed. Whereas, for $\mathcal M_\delta^s$ which corresponds to $\dim(S_{s'})\geq \min\{s'+1,d\},s'=2,...,d$, higher order curvature, such as torsion are needed.

   A slightly special case is $s'=d,$ where finite boundedness of the maximal operator holds. In the planar case, this is Bourgain's circular maximal theorem \cite{bourgain1984spherical}. The sharp range estimate in $\R^3$ is observed independently by \cite{ko2021maximalestimatesaveragesspace} and \cite{beltran2021sharplpboundshelical} and more recently, Gan, Maldague, and Oh \cite{gan2025sharplocalsmoothingestimates} proved this in $\R^4$ with the same type routine method described in this note.
\end{remark}

We summarize the results of this note with the following table \ref{Tab:SummaryOfMainResults}.
\begin{table}[h]
    \centering
    \begin{tabular}{cccccc}
    \toprule
        dim. of sup. &  maxl. op. & maxl. est. & sharp exp. & sharp $p_d$\\ \midrule
        $s=d$ & $\mathcal{M}^d_\delta:\R^{1}\to \R_+$ & \makecell{$\norm{\mathcal{M}^d_\delta}_{p\to p}=O(\delta^{-\eps-\frac{d-2}{p}})$,\\$p\geq 3$. {Section \ref{ProjectionArgument}}} &\makecell{Yes\\ {Section \ref{SharpExponent}}} & \makecell{Yes\\{Section \ref{SharpRange}}} \\
        $s\in\{1,...,d-1\}$ &  $\mathcal{M}^s_\delta:\R^{s'}\to \R_+$ & \makecell{$\norm{\mathcal{M}^s_\delta}_{p\to p}=O(\delta^{-\alpha(s)})$,\\$p\geq 4d-2$. {Section \ref{SobolevEmbedding}}} & \makecell{Yes\\ {Section \ref{SharpExponent}}} & \makecell{Possibly No\\{Section \ref{SharpRange}}} \\\bottomrule
    \end{tabular}
    \caption{summary of main results}  \label{Tab:SummaryOfMainResults}
\end{table}

We also summarize the methods that we applied and the structure of the note using Figure \ref{fig:summaryofmethodandstructure}. 
\begin{figure}[h]
    \centering

\tikzset{every picture/.style={line width=0.75pt}} 

\begin{tikzpicture}[x=0.75pt,y=0.75pt,yscale=-1,xscale=1]

\draw    (306.66,112.4) -- (306.66,76.34)(309.66,112.4) -- (309.66,76.34) ;
\draw [shift={(308.16,69.34)}, rotate = 90] [color={rgb, 255:red, 0; green, 0; blue, 0 }  ][line width=0.75]    (10.93,-4.9) .. controls (6.95,-2.3) and (3.31,-0.67) .. (0,0) .. controls (3.31,0.67) and (6.95,2.3) .. (10.93,4.9)   ;
\draw    (190.43,248.13) -- (218.54,176.37)(193.22,249.22) -- (221.33,177.46) ;
\draw [shift={(222.49,170.4)}, rotate = 111.39] [color={rgb, 255:red, 0; green, 0; blue, 0 }  ][line width=0.75]    (10.93,-4.9) .. controls (6.95,-2.3) and (3.31,-0.67) .. (0,0) .. controls (3.31,0.67) and (6.95,2.3) .. (10.93,4.9)   ;
\draw    (438.48,249.39) -- (404.9,181.01)(441.17,248.07) -- (407.59,179.69) ;
\draw [shift={(403.16,174.07)}, rotate = 63.85] [color={rgb, 255:red, 0; green, 0; blue, 0 }  ][line width=0.75]    (10.93,-4.9) .. controls (6.95,-2.3) and (3.31,-0.67) .. (0,0) .. controls (3.31,0.67) and (6.95,2.3) .. (10.93,4.9)   ;

\draw (243,47) node [anchor=north west][inner sep=0.75pt]   [align=left] {dimension estimates dimension estimates Corollary \textbackslash ref\{dimest\}};
\draw (241.33,125) node [anchor=north west][inner sep=0.75pt]   [align=left] {maximal estimates for };
\draw (322,74.33) node [anchor=north west][inner sep=0.75pt]   [align=left] {{\footnotesize standard argument}\\{\footnotesize Section \textbackslash ref\{Fthm2toThm1\} }};
\draw (364,199) node [anchor=north west][inner sep=0.75pt]  [font=\footnotesize] [align=left] {``Sobolev embbeding"\\Section \ref{SobolevEmbedding}};
\draw (149,199) node [anchor=north west][inner sep=0.75pt]  [font=\footnotesize] [align=left] {projection argument\\Section \ref{ProjectionArgument}};
\draw (128,254.33) node [anchor=north west][inner sep=0.75pt]   [align=left] {Wollf's circular \\maximal Theorem};
\draw (379.67,253.67) node [anchor=north west][inner sep=0.75pt]   [align=left] {Local smoothing \\for average over curves};
\draw (206.33,142.73) node [anchor=north west][inner sep=0.75pt]    {$s=d$};
\draw (397.33,123.73) node [anchor=north west][inner sep=0.75pt]    {$\mathcal{M}_{\delta }^{s}$};
\draw (353.33,142.73) node [anchor=north west][inner sep=0.75pt]    {$s=1,...,d-1$};

\end{tikzpicture}

    \caption{summary of methods and structure of the note}
    \label{fig:summaryofmethodandstructure}
\end{figure}
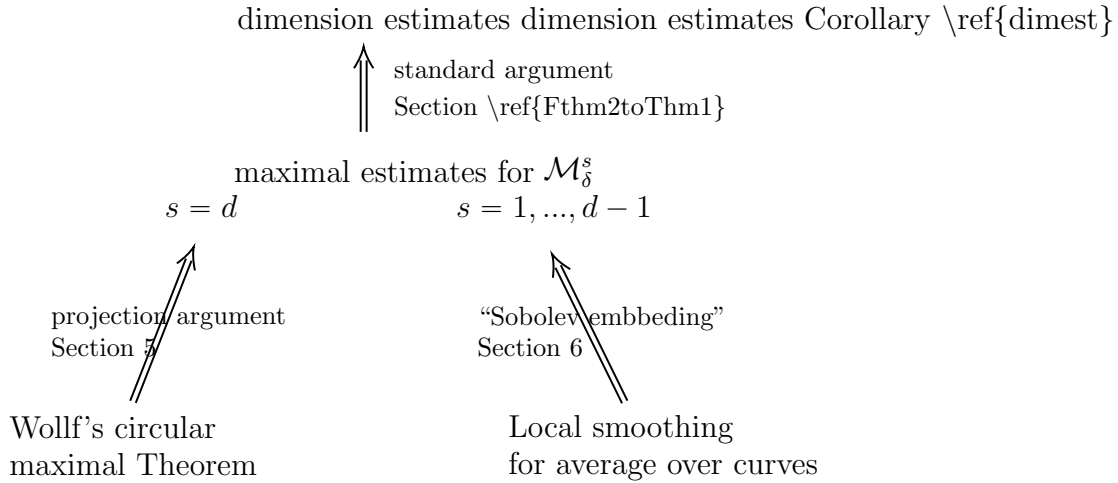
\noindent To discuss the sharpness, in Section \ref{sectangency}, some preliminary geometric facts are provided.

\begin{remark}
    In the planar case, the argument we applied to Theorem \ref{MaxEst}, $s=1,...,d-1$ still works. However, this method relying on local smoothing \cite{Ko_Lee_Oh_2023} cannot provide sharp range maximal estimate even with the sharp range local smoothing by Guth, Wang, and Zhang \cite{GuthWangZhang}.
    
\end{remark}


\section{From Theorem \ref{MaxEst} to Corollary \ref{dimest}}\label{Fthm2toThm1}

We first prove the following standard result passing from maximal estimates, Theorem \ref{MaxEst} to dimension estimates of BRK type sets, Corollary \ref{dimest}. A similar proof can be found in \cite[Lemma 11.9]{WolffBook}.

\begin{proposition}\label{MaxToDim}
\begin{enumerate}
    \item[(i)] $s\in \{2,...,d\}$. Assume for some finite $p$, there is $\alpha>0$ such that   
    \begin{equation}\label{bdd1}
    \norm{\M_\delta^s f}_{L^p(I_1^{s'-1}\times I_2)}\lesssim \delta^{-\alpha}\norm{f}_{L^p(\R^d)},
\end{equation}
then 
\begin{equation*}
    \dim(S_{s'})\geq d-p\alpha.
\end{equation*}

\item[(ii)]Assume for some finite $p$, there is $\alpha<0$ \begin{equation}\label{bdd2}
    \norm{\M_\delta^1 f}_{L^p(I_1^{d-1}\times I_2)}\lesssim \delta^{-\alpha}\norm{f}_{L^p(\R^d)},
\end{equation}
then 
\begin{equation*}
    \L^d(S_d)>0\Rightarrow \dim(S_d)=d.
\end{equation*}
\end{enumerate}
\end{proposition}

To deduce Corollary \ref{dimest}, for $s=2,...,d,$ by item $(i)$, 
\begin{equation*}
    \dim(S_{s'})\geq d-p\left(\eps+\frac{s-2}p\right)=d-s+2-p\eps\to s'+1,\text{ as } \eps\to 0.
\end{equation*}
When $s=1$, the conclusion is derived directly from item $(ii)$.

Before proving Proposition \ref{MaxToDim}, we first introduce the prerequisites of Hausdorff dimension.
\begin{definition}[Hausdorff measure]
    For a nonnegative real number $s\in(0,d]$, the $s$-dimensional \emph{Hausdorff measure} of a set \( E \subset \mathbb{R}^d \) 
    is defined as:

\[
\mathcal{H}^s(E) = \lim_{\delta \to 0} \inf \left\{ \sum_{i=1}^\infty (\text{diam}(U_i))^s : E \subset \bigcup_{i=1}^\infty U_i \text{ and } \text{diam}(U_i) < \delta \right\}
\]

where $U_i$ is an open ball and \( \text{diam}(U_i) \) denotes its diameter. 
\end{definition}
For some fixed $\delta_0,$ we also use $\mathcal H_{\delta_0}^s$ to denote 
\begin{equation}\label{ForHausdorffContnent}
\mathcal{H}^s_{\delta_0}(E) = \inf \left\{ \sum_{i=1}^\infty (\text{diam}(U_i))^s : E \subset \bigcup_{i=1}^\infty U_i \text{ and } \text{diam}(U_i) < \delta_0 \right\}.
\end{equation}
It is elementary to see that if we treat $\mathcal{H}^s_{\delta_0}(E)$ as a function of $\delta_0$, then it is a monotonically decreasing function. Therefore, $\mathcal{H}^s(E)\geq \mathcal{H}^s_{\delta_0}(E),\forall \delta_0>0.$

Then we can introduce the Hausdorff dimension:

\begin{definition}[Hausdorff dimension]\label{HausDim}
    The \emph{Hausdorff dimension} of a set \(E\) in $\R^d$ is defined as:

\[
\dim_{\mathbb H}(E) = \inf \{s \geq 0 \mid \mathcal{H}^s(E) = 0\} = \sup \{s \geq 0 \mid \mathcal{H}^s(E) = \infty\}
\]

where \(\mathcal{H}^s(E)\) denotes the \(s\)-dimensional Hausdorff measure of \(E\). 
\end{definition}
\begin{remark}\label{HausdorffCotentAndDim}
    When $\delta_0=\infty$ in \eqref{ForHausdorffContnent}, $H^s_\infty$ is called \textit{$s$-dimensional Hausdorff content of $E$.} We will use the following fact
    \begin{equation*}
        H^s_\infty(E)<\infty\Rightarrow \dim_{\mathbb H}(E)\leq s.
    \end{equation*}
\end{remark}
Unless otherwise, we will always abbreviate \(\dim_{\mathbb H}(E)\) as $\dim(E)$ as before. Now we can start by proving Proposition \ref{MaxToDim}.

\begin{proof}We start with item $(i)$. Fix $s\in \{2,...,d\}$ and $\alpha>0$. Suppose $S_{s'}$ is a $s'$-parameter moment BRK type set. If we can prove that for all $0<\eps\ll 1$, such that 
    \begin{equation*}
        \mathcal{H}^{d-p\alpha-\eps}_{\frac{1}{100}}(S_{s'})\geq C_\eps,
    \end{equation*}
    then by the monotonicity, we have 
    \begin{equation*}
        \HH^{d-p\alpha-\eps}(S_{s'})\geq   \mathcal{H}^{d-p\alpha-\eps}_{\frac{1}{100}}(S_{s'})\geq  C_\eps,
    \end{equation*}
    and this implies $\dim(S_{s'})\geq d-p\alpha-\eps$ by Definition \ref{HausDim}. Then by a similar limiting argument, we can obtain $\dim(S_{s'})\geq d-p\alpha$. Note that in this resoning, we need $\alpha$ to be non-negative.

    Take an arbitrary countable covering of $S_{s'}.$ Denote it as $\bigcup_jB(\mathbf x_j,r_j)$, where $r_j<\frac{1}{100}$. It suffices to prove 
    \begin{equation*}
        \sum_jr_j^{d-p\alpha-\eps}\geq C_\eps .
    \end{equation*}

    We first arrange the balls $B(\mathbf x_j,r_j)$ with respect to their sizes. By dyadic decomposition, define
    \begin{equation*}
        J_k:=\left\{B_j=B(\mathbf x_j,r_j): 2^{-k-1}<r_j\leq 2^{-k}\right\},~k=6,7,...
    \end{equation*}
    Note that each ball $B(\mathbf x_j,r_j)$ in the covering lies in a unique block $J_k$ and the blocks:
    \begin{equation*}
        J_6,~J_7,~J_8,~J_9,~J_{10},...
    \end{equation*}
    are disjoint. The sizes of the balls contained in each block are comparable and the radii of the balls in different blocks decreases as $k$ increase. Then we rewrite the summation
    \begin{equation}\label{summation}
        \sum_jr_j^{d- p\alpha-\eps}=\sum_k\sum_{B\in J_k}r(B)^{d-p\alpha-\eps}\approx \sum_k2^{-k(d-p\alpha-\eps)}\#J_k.
    \end{equation}
    Therefore, it suffices to estimate each term on the right-hand side from below.

      For each $(\underline{\mathbf x},r)\in I_1^{s'-1}\times I_2$, there is a moment curve of the form $H(\cdot,\underline{\mathbf x},r)\subseteq S_{s'}$ (If in $S_{s'}$. As a result, there are more than one moment curves satisfying this, then choose one  arbitrary.), there must be a block $J_k$ whose members collectively intersect ``most" part of the moment curve $H(\cdot,\underline{\mathbf x},r)$. In other words, there exists $k_0=k_0(\underline{\mathbf x},r)$ (depending on $(\underline{\mathbf x},r)$)) such that 
    \begin{equation}\label{mcbintersec}
       \mathcal L^1\Bigl( H(\cdot,\underline{\mathbf x},r)\bigcap \Bigl(\cup_{B\in J_{k_0}}B\Bigr) \Bigr)>\frac{1}{100k_0^2}.
    \end{equation}
    This can be proven by pigeonholing. Indeed, if no such $k_0$ exists, then for all $k$,
    \begin{equation*}
         \mathcal L^1\Bigl(H(\cdot,\underline{\mathbf x},r)\bigcap \Bigl(\cup_{B\in J_k}B\Bigr) \Bigr)\leq \frac{1}{100}\frac{1}{k^2}.
    \end{equation*}
    This implies
    \begin{align*}
        \frac{1}{10}\leq \mathcal{L}^1(H(\cdot,\underline{\mathbf x},r))&=\mathcal L^1\Bigl(H(\cdot,\underline{\mathbf x},r)\bigcap \Bigl(\cup_{j=6}^\infty B(\mathbf x_j,r_j)\Bigr) \Bigr)\\
        &\leq \sum_k   \mathcal L^1\Bigl(H(\cdot,\underline{\mathbf x},r)\bigcap \Bigl(\cup_{B\in J_k}B\Bigr) \Bigr)\\
        &\leq \frac{1}{100}\sum\frac{1}{k^2}\\
        &<\frac{1}{10},
    \end{align*}
    which is a contradiction.

    So for each $(\underline{\mathbf x},r)\in I_1^{s'-1}\times I_2$, there is a $k_0=k_0((\underline{\mathbf x},r))$ such that equation \eqref{mcbintersec} holds. Let $R_k:=\{(\underline{\mathbf x},r)\in I_1^{s'-1}\times I_2:k_0(\underline{\mathbf x},r)=k \}.$ Therefore, the union of $R_k$ is $I_1^{s'-1}\times I_2$. 
    Let $f_k$ be the characteristic function of $\bigcup_{B\in J_k}10B$ and take $\delta=2^{-k}$.  On one hand,
    \begin{equation}\label{lastline}
        \begin{aligned}
        ||\M^s_{2^{-k}}f_k||^p_{L^p{(I_1^{s'-1}\times I_2)}}&=\int_{I_1^{s'-1}}\int_{I_2}\left| 
\sup_{\overline{\mathbf x}\in \R^{s}}\frac{1}{\mathcal L^d(H_\delta(\overline {\mathbf x},\underline{\mathbf x},r))}\int_{H_\delta(\overline {\mathbf x},\underline{\mathbf x},r)}f_k \right|^pdrd\underline{\mathbf x}\\
&\geq \int\int_{R_k}\left| 
\sup_{\overline{\mathbf x}\in \R^{s}}\frac{1}{\mathcal L^d(H_\delta(\overline {\mathbf x},\underline{\mathbf x},r)}\int_{H_\delta(\overline {\mathbf x},\underline{\mathbf x},r)}f_k \right|^pdrd\underline{\mathbf x}\\
&\gtrsim \mathcal L^{s'}(R_k)k^{-2p}.
        \end{aligned}
    \end{equation}
The last ``$\gtrsim$" is because $|H_\delta(\mathbf x,r)\bigcap \left(\bigcup_{B\in J_k} 10B\right)|\gtrsim \delta^{d-1} \times k^{-2}=2^{-k(d-1)}k^{-2}$ by \eqref{mcbintersec}. Indeed, since \eqref{mcbintersec} holds, if we enlarge each $B\in J_k$ by $10$ times, then the intersection $H_\delta(\cdot,\underline{\mathbf x},r)\cap \bigcup_{B\in J_k} 10B$ contains some subtubes of $H_\delta(\cdot,\underline{\mathbf x},r)$ with total length $\sim \frac{1}{k^2}$ and width $\delta$. So 
\begin{equation*}
    \int_{H_\delta(\mathbf x,r)}f_k=\mathcal L^d\left( H_\delta(\mathbf x,r)\bigcap \left(\bigcup_{B\in J_k}10B\right)\right)\gtrsim \delta^{d-1} \times k^{-2}=2^{-k(d-1)}k^{-2}.
\end{equation*}
$2^{-k(d-1)}$ cancels with $\frac{1}{\mathcal L^d(H_\delta(\mathbf x,r))}$ and we obtain the last line of \eqref{lastline}.

On the other hand, by the boundedness \eqref{bdd1},
\begin{equation*}
     ||\M^s_{2^{-k}}f_k||^p_p\lesssim 2^{p\alpha k}||f_k||_p^p\lesssim 2^{p\alpha k}\#J_k2^{-kd}.
\end{equation*}

    Combining the lower bound and the upper bound, we obtain a lower bound for $\#J_k$:
    \begin{equation*}
        \mathcal L^{s'}(R_k)k^{-2p}\lesssim 2^{-(d-p\alpha)k}\#J_k.
    \end{equation*}

      Plugging these back the summation \eqref{summation}, we have there is $C_\eps>0$ such that
    \begin{equation*}
        \sum_k\#J_k2^{-k(d-p\alpha-\eps)}\gtrsim \sum_k 2^{k\eps} k^{-2p}\mathcal L^2(R_k)\geq C_\eps \sum_k  \mathcal L^{d-1}(R_k')\gtrsim C_\eps.
    \end{equation*}
    This concludes
the proof of item $(i)$. 

We now proceed with the proof of item $(ii)$. Since
    \begin{equation*}
        \norm{\mathcal M^1_\delta f}_{L^p(I_1^{d-1}\times I_2)}\lesssim \norm{f}_{L^p(\R^d)},
    \end{equation*}
    By taking $f=\chi_{S_{d}}$, we can obtain $\norm{f}_{L^p(\R^d)}=|S_d|^{1/p}\gtrsim  \norm{\mathcal M^1_\delta \chi_{S_d}}_{L^p}\approx 1.$
   
This concludes the proof of Proposition \ref{MaxToDim}. 
\end{proof}

\section{Tangency of moment curves}\label{sectangency}

In this section, we discuss the conditions and properties for the exact and almost tangency of two moment curves of the form \eqref{mcset}. We first provide some basic observations about the moment curves. 
\begin{lemma}[comparable $\delta$-neighborhoods] \label{deltanbhdscomparable} Assume $H_\delta(\mathbf x,r)$ and $H_\delta(\mathbf x',r')$ are two $\delta$-neighborhoods of moment curves. If $|(\mathbf x,r)-(\mathbf x',r')|\leq \delta$, then there is $c$ and $C$ such that 
\begin{equation*}
    H_{c\delta}(\mathbf x,r)\subseteq H_\delta(\mathbf x',r')\subseteq CH_{C\delta}(\mathbf x,r).
\end{equation*} 
\end{lemma}
\begin{proof}
    By symmetry, it suffices to prove the right inclusion relation. For any $\mathbf y\in H_\delta(\mathbf x',r')$ there is some $t_0\in [-1,1]$ such that 
    \begin{equation*}
        |\mathbf y-\mathbf x'-r'\gamma(t_0)|\leq \delta.
    \end{equation*}
    By triangle inequality,
    \begin{align*}
         |\mathbf y-\mathbf x-r\gamma(t_0)|&\leq   |\mathbf y-\mathbf x'-r'\gamma(t_0)|+  |\mathbf x'+r'\gamma(t_0)-\mathbf x-r\gamma(t_0)|\\
         &\leq \delta+2\delta\\
         &\leq 3\delta.&\qedhere
    \end{align*}
\end{proof}
This lemma tells us that $\delta$-perturbation to a moment curve is acceptable and the $\delta$-neighborhood of a moment curve is roughly the union of such perturbations.

\subsection{Exact tangency} We first discuss the exact tangency of two moment curves. In what follows, ``two moment curves" always means ``two distinct moment curves".
\begin{definition}[tangency of two moment curves]
    Consider two moment curves $H(\mathbf x,r)$ and $H(\mathbf x',r')$ with $(\mathbf x,r)\neq (\mathbf x',r')$. We will say $H(\mathbf x,r)$ and $H(\mathbf x',r')$ are \textit{tangent at a point $\mathbf y$} if there are $t,t'\in[-1,1]$ such that 
\begin{equation*}
    \mathbf y=\mathbf x +r\gamma(t)=\mathbf x'+r'\gamma(t'),
\end{equation*}
and the tangency vectors to $H(\mathbf x,r)$ and $H(\mathbf x',r')$ at $\mathbf y$ are parallel.
\end{definition}


\begin{proposition}[characterization for tangency of two moment curves]\label{chafortan}

    Assume $H(\mathbf x,r)$ and $H(\mathbf x',r')$ are two moment curves, where $\mathbf x=(x_1,...,x_d)$ and $\mathbf x'=(x_1',...,x_d')$. 
    \begin{enumerate}
        \item $H(\mathbf x,r)$ and $H(\mathbf x',r')$ can be tangent only if $r\neq r'.$
        \item  When $r\neq r'$, define 
    \begin{equation*}
        \Delta_i((\mathbf x,r),(\mathbf x',r'))=|(x_i-x_i')(r-r')^{i-1}-(x_1-x_1')^i|,\ i=2,...,d.
    \end{equation*}
   Then $H(\mathbf x,r)$ and $H(\mathbf x',r')$ are tangent at a point $\mathbf x+r\gamma(t)=\mathbf x'+r'\gamma(t') $ \textbf{if and only if} 
    \begin{equation}\label{conditions}
            \Delta_i=0 \text{, for all } i=2,...,d \text{ and } t=t'=\frac{x_1-x_1'}{r'-r}\in [-1,1].
    \end{equation}
    \end{enumerate}

\end{proposition}
\begin{proof}
       We will prove \textit{(2)} and observe along the way that \textit{(1)} holds.
       
       \noindent\textit{Necessity.} If the two moment curves intersect then for some point $t,t'\in [-1,1]$, we have $\mathbf x'+r'\gamma(t')=\mathbf x+r\gamma(t)$, more precisely, 

\begin{equation}
    \left\{
        \begin{aligned}\label{systems}
            x_1'+r't' &= x_1+rt \\
            x_2'+r't'^2 &= x_2+rt^2 \\
            \vdots \\
            x_d'+r't'^d &= x_d+rt^d\\
        \end{aligned}
    \right. 
\end{equation}

    Moreover, if they are also tangent here, then
    \begin{equation}\label{tangentcondition}
        r'(1,2t',...,dt'^{d-1})\parallel r(1,2t,...,dt^{d-1}).
    \end{equation}
     From the limitations on the first two entries we obtain $\frac{r'}{r}=\frac{2r't'}{2rt}$. This implies that $t'=t.$ 
     
     If we also have $r=r'$, then take it back to equation \eqref{systems}, we deduce that $\mathbf x=\mathbf x'$ which contradicts with the assumption that the two moment curves are distinct. This concludes the proof of part \textit{(1)}. 

     Since we have $r\neq r',$ plug it back to equation \eqref{systems}, we obtain the conditions \eqref{conditions}.
     

\noindent\textit{Sufficiency.} Conversely, if the conditions \eqref{conditions} hold then the system of equation \eqref{systems} is satisfied with $t=t'\in [-1,1]$. Thus the two moment curves must intersect at some point where their parameters are the same. The tangent vectors to $H(\mathbf x,r)$ and $H(\mathbf x',r')$ at the intersection point are $r\gamma'(t)$ and $r'\gamma'(t')$ respectively. If $t=t'$ as assumed in \eqref{conditions}, these two tangent vectors are parallel.
\end{proof}
From Proposition \ref{chafortan}, we can also see that two moment curves can be tangent at no more than one point. This follows from the expression for $t,t'$ in \eqref{conditions}. Furthermore, from the first two equations of system \eqref{systems}, we also have the following 
\begin{lemma}[number of intersections]\label{numberofintersection}
    The number of points of intersection of two moment curves is at most $2$.
\end{lemma}


\subsection{Almost tangency}
For the almost tangency case, We have the following conclusions similar to those of Kolasa and Wolff \cite[Lemma 3.3]{Kolasa1999OnSV}, Wolff \cite[Lemma 11.20]{WolffBook} and Pramanik, Yang and Zahl \cite[Lemma 3.8]{pyz}.
We first give the following definitions. 

\begin{definition} For two moment curves $H(\mathbf x,r)$ and $H(\mathbf x',r')$, define 
\begin{equation}\label{distanceinparaspace}
         \overline{d}((\mathbf x,r),(\mathbf x',r'))=|x_1-x_1'|+|x_2-x_2'|+|r-r'|.
    \end{equation}
     We also define 
    \begin{equation}\label{delta2}
        \overline{\Delta}((\mathbf x,r),(\mathbf x',r'))=\frac{|(x_2-x_2')(r-r')-(x_1-x_1')^2|}{|x_1-x_1'|+|x_2-x_2'|+|r-r'|}=\frac{\Delta_2}{ \overline{d}}.
    \end{equation}
\end{definition}
    These $ \overline{d}$ and $\overline{\Delta}$ are used to describe the tangency in the plane. Similar to the computations in Proposition \ref{chafortan}, it is easy to check that when $\overline{\Delta}((\mathbf x,r),(\mathbf x',r'))=0$, the projections of $H(\mathbf x,r)$ and $H(\mathbf x',r')$ to the plane which are the two parabolas, are tangent.
    \begin{remark}\label{rmk2}
        The reason why we need to divide $\Delta_2$ by $\overline{d}$ is homogeneity. Heuristically, the degree of $\Delta_2$ is $2$ while $ \overline{d}$ is $1$. We will see from the proof of the following Proposition \ref{prointersectionvolum} that it is necessary to maintain the degree of $\overline{\Delta}$ and $ \overline{d}$ the same. 
    \end{remark}

\begin{proposition}[intersection volume]\label{prointersectionvolum}
     Assume $H(\mathbf x,r)$ and $H(\mathbf x',r')$ are two moment curves where $\mathbf x=(x_1,x_2,...,x_d),\ \mathbf x'=(x_1',x_2',...,x_d')$. 
    Then 
        $H_\delta(\mathbf x,r) \bigcap H_\delta(\mathbf x',r')$ is contained in a union of $\delta$-neighborhood(s) of at most two arc(s) of $H(\mathbf x,r)$ with total length $\lesssim\frac{\delta}{\sqrt{(\delta+\overline{\Delta})(\delta+\overline{d})}}$ hence with measure $\lesssim \frac{\delta^d}{\sqrt{(\delta+\overline{\Delta})(\delta+\overline{d})}}$. The implicit constants are independent of the choice of $H(\mathbf x,r)$, $H(\mathbf x',r')$ and $\delta.$
\end{proposition}

\begin{proof}
    We first prove the conclusion in the plane from which we can deduce the higher dimensional cases.

    We can assume that one of the parabolas is $H(\mathbf 0,1)$ and the other one is $H(\mathbf x,r)$. Since a non-degenerate affine transformation will not affect the intersection relation of two moment curves and furthermore, since our centers and radii are limited in a compact set, the intersection area is changed at most some constant multiples. In addition, we can also assume that the two parabolas intersect. This is because, if $H_\delta(\mathbf 0,1)\cap H_\delta(\mathbf x,r)\neq \emptyset$ (which is the case we care about, otherwise the estimate is trivial), then through a translation, we can find a new parabola $H(\mathbf x',r)$ such that 
    \begin{equation*}
        H(\mathbf x',r)\bigcap  H(\mathbf 0,1)\neq \emptyset,
    \end{equation*}
    and $|\mathbf x-\mathbf x'|\leq \delta$. Therefore, via triangle inequality, we have 
    \begin{equation*}
         H_\delta(\mathbf x,r)\bigcap  H_\delta(\mathbf 0,1)\subseteq   H_{(100d)^d\delta}(\mathbf x',r)\bigcap  H_{(100d)^d\delta}(\mathbf 0,1).
    \end{equation*}
    So it suffices to analyze the intersection case. 

    The proof is a direct computation. The parabolas are graphs of the following two functions over $[-1,1]$,
    \begin{equation*}
        f_1(t)=t^2,\quad f_2(t)=x_2+\frac{(t-x_1)^2}{r}.
    \end{equation*}
    The difference between the two quadratic functions is 
    \begin{equation*}
        h(t)=t^2-x_2-\frac{(t-x_1)^2}{r}=(1-\frac{1}{r})t^2+\frac{2x_1}{r}t-(x_2+\frac{x_1^2}{r}).
    \end{equation*}
    Since we assume the two parabolas intersect, this means the discriminant is nonnegative. By Lemma \ref{numberofintersection}, it suffices to consider one of the solutions, denoted as $t_0,$ the other one is similar. Assume $s$ satisfies 
    \begin{equation*}
        |h(t_0+s)|= (100d)^d\delta.
    \end{equation*}
    Our goal is to give an upper bound for the smaller $|s|$. Note that $ |h(t_0+s)|$ is also a quadratic function. It is easy to check by direct computations that 
    \begin{equation*}
        |s|\lesssim_d \frac{\delta}{\sqrt r \sqrt{(r-1)x_2+x_1^2}}  \lesssim \frac{\delta}{\sqrt{(\overline{\Delta}((\mathbf x,r),(\mathbf 0,1))+\delta)(\overline{d}((\mathbf x,r),(\mathbf 0,1))+\delta)}}
    \end{equation*}
    which is desired. In the second $\lesssim$, we used the definition \eqref{delta2} of $\Delta_2$ and Remark \ref{rmk2} below that definition. There are at most two $t_0$'s, so the number of such intervals is at most two.

    For the higher dimensional case, define the orthogonal projection $ \pi_2:\R^d\to   \R^2$
    \begin{equation*}
       \pi_2(x_1,x_2,...,x_d) := (x_1,x_2).    
    \end{equation*}
     When $H_\delta(\mathbf x,r)\bigcap H_\delta(\mathbf x',r')=\emptyset,$ the estimate holds trivially, so we assume the intersection is nonempty.  
    Clearly, the projections of the moment curves
      \begin{equation*}
       \pi_2(H(\mathbf x,r))=\{(x_1,x_2)+r(t,t^2):t\in[-1,1]\}=:P(x_1,x_2,r)=P
    \end{equation*}
    and
    \begin{equation*}
        \pi_2(H(\mathbf x',r'))=\{(x_1',x_2')+r'(t',t'^2):t'\in[-1,1]\}=:P(x_1',x_2',r')=P'
    \end{equation*}
    are two parabolas in the plane. We are going to prove $\pi_2(H_\delta(\mathbf x,r))$ and $\pi_2(H_\delta(\mathbf x',r'))$ are $\delta$-neighborhoods in the plane of the the parabolas $P_\delta(x_1,x_2,r)=P_\delta$ and $P_\delta(x_1',x_2',r')=P'_\delta$ respectively. That is, we prove the following  
    \begin{equation}\label{p=pi}
        P_\delta = \pi_2(H_\delta(\mathbf x,r)).
    \end{equation}

    We first prove $\pi_2(H_\delta(\mathbf x,r))\subseteq P_\delta$. For each $\mathbf y\in H_\delta(\mathbf x,r)$ there is a $\mathbf p\in H(\mathbf x,r)$ such that 
    \begin{equation*}
        |\mathbf p-\mathbf y|\leq \delta.
    \end{equation*}
    Since $\pi_2$ is Lipschitz with constant $1$, this implies that for the point $\pi_2(\mathbf p)\in    \pi_2(H(\mathbf x,r))$,
    \begin{equation*}
        |\pi_2(\mathbf y)-\pi_2(\mathbf p)|\leq |\mathbf p-\mathbf y|\leq \delta.
    \end{equation*}
    Since $\pi_2(\mathbf p)\in P$, $\pi_2(H_\delta(\mathbf x,r))\subseteq P_\delta$. For the other inclusion $\pi_2(H_\delta(\mathbf x,r))\supseteq P_\delta$, assume $(y_1,y_2)\in P_\delta$, then there is a point on the moment curve $\mathbf p=(p_1,p_2,...,p_d)\in H(\mathbf x,r)$ such that 
    \begin{equation}\label{ClosePoint}
        |(y_1,y_2)-\pi_2(\mathbf p)|=|(y_1,y_2)-(p_1,p_2)|\leq \delta.
    \end{equation}
    By \eqref{ClosePoint},  $|(y_1,y_2,p_3,...,p_d)-\mathbf p|\leq \delta.$ Hence $(y_1,y_2,p_3,...,p_d)\in H_\delta(\mathbf x,r)$ whose projection under $\pi_2$ is exactly $(y_1,y_2).$ Therefore, $\pi_2(H_\delta(\mathbf x,r))\supseteq P_\delta$.
  
    With this fact combined with a similar argument, we can deduce that
    \begin{equation}\label{inc}
        H_\delta(\mathbf x,r)\bigcap H_\delta(\mathbf x',r')=\pi_2^{-1}\big( P_\delta \cap P'_\delta\big)\bigcap  H_\delta(\mathbf x,r)\bigcap H_\delta(\mathbf x',r').
    \end{equation}
    Indeed, since
    \begin{equation*}
        \pi_2( H_\delta(\mathbf x,r)\bigcap H_\delta(\mathbf x',r'))\subseteq  \pi_2( H_\delta(\mathbf x,r))\bigcap  \pi_2( H_\delta(\mathbf x',r'))=P_\delta\bigcap P'_\delta.
    \end{equation*}
    This implies that 
    \begin{equation*}
        H_\delta(\mathbf x,r)\bigcap H_\delta(\mathbf x',r'))\subseteq \pi_2^{-1}\left(P_\delta\bigcap P'_\delta\right)
    \end{equation*}
which means the right inclusion of \eqref{inc} holds. Whereas the left inclusion is trivial.

    From the first part of the proof, we know that $P_\delta \bigcap P'_\delta$ is contained in a union of $\delta$-neighborhood(s) of arc(s) in the plane of  total length $\lesssim\frac{\delta}{\sqrt{(\delta+\overline{\Delta})(\delta+\overline{d})}}.$ Denote this union of the $\delta$-neighborhood(s) as $Q$. From \eqref{inc} and $P_\delta \bigcap P'_\delta\subseteq Q$, we have 
    \begin{equation*}
         H_\delta(\mathbf x,r)\bigcap H_\delta(\mathbf x',r')\subseteq\pi_2^{-1}(Q)\bigcap  H_\delta(\mathbf x,r).
    \end{equation*}
     The right-hand side is a union of at most two $\delta$-neighborhoods of $H(\mathbf x,r)$ with dimensions $\lesssim \frac{\delta}{\sqrt{(\delta+\overline{\Delta})(\delta+\overline{d})}}\times \delta\times...\times \delta= \frac{\delta}{\sqrt{(\delta+\overline{\Delta})(\delta+\overline{d})}}\times \delta^{d-1}$. 
    \end{proof}

    For two ``very different" moment curves that are exactly tangent, the intersection of their $\delta$-neighborhoods behaves locally like the intersection of two $\delta$-tubes. That is,
    \begin{proposition}\label{exacttangencyvolume}
          Assume that $H_\delta(\mathbf x,r)$ and $ H_\delta(\mathbf x',r')$ as before. 
          \begin{enumerate}
              \item  If $\overline{d}((\mathbf x,r),(\mathbf x',r'))\gtrsim  1$ and the two moment curves are exactly tangent, then the intersection is comparable to a rectangle with dimensions $\sqrt \delta\times \delta\times...\times \delta=\sqrt \delta \times \delta^{d-1}$ hence,
        \begin{equation}\label{tangencyvolume1}
       \mathcal L^{d}\Bigl(H_\delta(\mathbf x,r)\bigcap H_\delta(\mathbf x',r')\Bigr) \sim  \delta^{d-\frac{1}{2}}.
    \end{equation}
    \item     If $\mathbf z\in B(\mathbf x,\frac{\delta}{1000d}),$ then a similar estimate also holds, i.e. 
    \begin{equation}\label{tangencyvolume2}
         \mathcal L^{d}\Bigl(H_\delta(\mathbf x',r')\bigcap H_\delta(\mathbf z,r)\Bigr) \sim  \delta^{d-\frac{1}{2}}.
    \end{equation}
    
          \end{enumerate}

    \end{proposition}
    \begin{proof}
         Since the point of tangency is unique, the intersection of $\delta$-neighborhoods is contained in a $\delta$-neighborhood of one piece of an arc. Since the two moment curves are exactly tangent, \eqref{conditions} holds which implies that $\overline{\Delta}=0$. Applying the upper bound given by Proposition \ref{prointersectionvolum},
         we obtain the $\lesssim$ part. 
         
         For the other side, we will prove there is a
    $\delta^{\frac{1}{2}}\times \delta\times...\times \delta$ tube that is contained in both $\delta$-neighborhoods of the two moment curves. Indeed, assume the two moment curves are exactly tangent at 
    \begin{equation}\label{tgt}
        \mathbf x+r\gamma(t)=\mathbf x'+r'\gamma(t), 
    \end{equation}
    Define line segment 
    \begin{equation*}
        L=\{\mathbf x+r\gamma(t)+sr\gamma'(t):s\in(0,\frac{1}{(100d)^d}\sqrt\delta)\}.
    \end{equation*}
    Consider its $\frac{1}{(100d)^d}\delta$-neighborhood $L_\delta$.   This is a $\delta^{\frac{1}{2}}\times \delta\times...\times \delta$ tube. We will prove $L_\delta\subseteq H_\delta(\mathbf x,r)$. For each $\mathbf y\in L_\delta$, there is a $s_o\in (0,\frac{1}{(100d)^d}\sqrt\delta)$ such that 
    \begin{equation}\label{hahaha}
        |\mathbf y-\mathbf x-r\gamma(t)-s_or\gamma'(t)|\leq \frac{1}{(100d)^d}\delta,
    \end{equation}
    then by triangle inequality, we have
    \begin{align*}
        |\mathbf y-\mathbf x&-r\gamma(t+s_o)|\\
        &\leq |\mathbf y-\mathbf x-r\gamma(t)-s_or\gamma'(t)| + |\mathbf x+r\gamma(t)+s_or\gamma'(t)-\mathbf x-r\gamma(t+s_o)| \\
        &=:I+II .
    \end{align*}
    The first term on the right hand side $I$ is under control by \eqref{hahaha}. For the second term $II$, by Taylor expansion, 
        \begin{align*}
      II&=|\mathbf x+r\gamma(t)+s_or\gamma'(t)-\mathbf x-r(\gamma(t)+\gamma'(t)s_o+\frac{\gamma''(t)}{2}s_o^2+o(|s_o|^2))|\\
       &=|r\frac{\gamma''(t)}{2}s_o^2+o(|s_o|^2)|\\
        &\leq 50d^2s_o^2\\
        &\leq \frac{\delta}{1000}.
    \end{align*}
    Therefore, we have 
    \begin{equation}\label{trianginq}
         |\mathbf y-\mathbf x-r\gamma(t+s_o)|\leq \frac{\delta}{50}.
    \end{equation}
    Combine this estimate with for $I$ and $II$, we deduce that $L_\theta\subseteq H_\delta(\mathbf x,r).$ By symmetry and the fact that $H_\delta(\mathbf x',r')$ is tangent to $H_\delta(\mathbf x,r)$ at the same point, we obtain  $L_\theta\subseteq H_\delta(\mathbf x',r').$ This concludes the proof of the first part .

    For the second item, it suffices to prove that $L_\delta\subseteq H_\delta(\mathbf z,r)$. By triangle inequality, we have for each $\mathbf y\in L_\delta$,
    \begin{equation*}
        |\mathbf y-\mathbf z-r\gamma(t+s_o)|\leq  |\mathbf y-\mathbf x-r\gamma(t+s_o)|+  |\mathbf x+r\gamma(t+s_o)-\mathbf z-r\gamma(t+s_o)|.
    \end{equation*}
    For the first term, apply \eqref{trianginq}, for the second term, apply $\mathbf z\in B(\mathbf x,\frac{\delta}{1000d})$. 
    \end{proof}

\section{Sharpness of the results}\label{examplesforsharpness}
In this section, we test several types of examples that will unravel various types of sharpness for the conclusions.

\subsection{Optimal exponent for $\delta$}\label{SharpExponent}
\begin{example} In this example, we will show that 
  \begin{lemma}\label{SharpExponentExample}
For and $p\geq 1$ and $s\in \{1,...,d\}$, if the maximal estimate
     \begin{equation*}
         \norm{\M_\delta^s f}_{L^p(I_1^{s'-1}\times I_2)}\lesssim \delta^{-\alpha}\norm{f}_p
     \end{equation*}
      holds, then $\alpha\geq \max \{\frac{s-2}{p},0\}$. 
  \end{lemma}
  \begin{proof}
    Assume $p\geq 1$. Define 
    \begin{equation}\label{SharpExampleForDimBound}       E^{s'}=\bigcup_{(\underline{\mathbf x},r)\in I_
        1^{s'-1}\times I_2}H(\mathbf 0,\underline{\mathbf x},r)=\bigcup_{(\underline{\mathbf x},r)\in I_
        1^{s'-1}\times I_2}\{(\mathbf 0,\underline{\mathbf x})+r\gamma(t):t\in [-1,1]\}
    \end{equation}
    and $E_\delta^{s'}$ is its $\delta$-neighborhood
    \begin{equation*}
        E_\delta^{s'}=\{\mathbf x\in \R^d: |\mathbf x-\mathbf y|\leq \delta, \mathbf y\in E \}.
    \end{equation*}
    Choose $f=\chi_{E_\delta^{s'}}$ to be the characteristic function of $E_\delta^{s'}$.  Our goal is to estimate the both sides of \eqref{max1}. We are going to prove 
    \begin{enumerate}
        \item $\M^s_\delta f(\underline{\mathbf x},r)\geq 1$ for all $(\underline{\mathbf x},r)\in I_
        1^{s'-1}\times I_2.$
        \item $\norm{f}_p\lesssim \min\{\delta^{(s-2)/p}$,1\}
    \end{enumerate}
     
    By the definition of $\M^s_\delta$, note that for each $(\underline{\mathbf x},r)\in I_
        1^{s'-1}\times I_2,$ $H(\mathbf 0,\underline{\mathbf x},r)\subseteq E,$ so $H_\delta(\mathbf 0,\underline{\mathbf x},r)\subseteq E_\delta^{s'}$ and $H_\delta(\mathbf 0,\underline{\mathbf{x}},r)\cap E_\delta^{s'}=H_\delta(\mathbf 0,\underline{\mathbf{x}},r)$. Therefore,
    \begin{align*}
        \M^s_\delta f(\underline{\mathbf{x}},r)&=\sup_{\overline{\mathbf x}\in\R^s}\frac{1}{\mathcal L^d( H_\delta(\mathbf x,r))}\int_{H_\delta(\mathbf x,r)}|f|\\
        &\geq \frac{1}{\mathcal L^d( H_\delta(\mathbf 0,\underline{\mathbf x},r))}\mathcal L^d(H_\delta(\mathbf 0,\underline{\mathbf x},r)\cap E_\delta^{s'})= 1,
    \end{align*}
    this concludes the proof of the first item.
    
    For the second item, take a $\delta$-net $A=\{(\underline{\mathbf x},r)_i\}_{i=1}^{O(N)}$ of $ I_
        1^{s'-1}\times I_2$. Then $N\sim \delta^{-s'}$. We claim
    \begin{equation}\label{inclusion}
        E_\delta^{s'}\subseteq \bigcup_{i=1}^NH_{10d\delta}(\mathbf 0,(\underline{\mathbf x},r)_i).
    \end{equation}
    Indeed, for each $\mathbf y\in E_\delta^{s'}$ there exist $(\underline{\mathbf x}^0,r^0)\in I_1^{s'-1}\times I_2,t_0\in[-1,1]$ such that 
    \begin{equation*}
        |\mathbf y-(\mathbf 0,\underline{\mathbf x}^0)-r^0\gamma(t_0)|\leq \delta. 
    \end{equation*}
    Since $A$ is the $\delta$-net of $I_1^{s'-1}\times I_2$, there is a $(\underline{\mathbf x},r)\in A$ such that $|(\underline{\mathbf x}^0,r^0 )-(\underline{\mathbf x},r)|\leq \delta$,
    then 
    \begin{align*}
        |\mathbf y-(\mathbf 0,\underline{\mathbf x})-r\gamma(t_0)|&\leq |\mathbf y-(\mathbf 0,\underline{\mathbf x}^0)-r^0\gamma(t_0)|+|(\mathbf 0,\underline{\mathbf x}^0)+r^0\gamma(t_0)-(\mathbf 0,\underline{\mathbf x})-r_{0}\gamma(t_0)|\\
        &\leq \delta+ |\underline{\mathbf{x}}^0-\underline{\mathbf x}| + |r-r_0||\gamma(t_0)|\\
        &\leq 2\delta+d\delta\\
        &\leq 10d\delta.
    \end{align*}
    Therefore, by \eqref{inclusion} and $N=O(\delta^{-s'})$, we have 
    \begin{equation*}
        \norm{f}_p^p=\norm{\chi_{E_\delta^{s'}}}_p^p=\mathcal L^{d}(E_\delta^{s'})\leq 
        \begin{cases}
            \sum_{i=1}^N\mathcal L^{d}(H_{10\delta d}(\mathbf 0,(\underline{\mathbf x},r)_i))\lesssim \delta^{s-2},&s=2,...,d,\\
            1,&s=1.
        \end{cases}     
        \end{equation*} 
    In order for \eqref{max1} to be true, therefore
    \begin{equation*}
        1\sim  \norm{\M_\delta^s f}_{L^p(I_1^{s'-1}\times I_2)}\lesssim \delta^{-\alpha}\norm{f}_p\lesssim \begin{cases}
            \delta^{-\alpha+\frac{s-2}{p}},&s=2,...,d,\\
            \delta^{-\alpha},&s=1.
        \end{cases}
    \end{equation*}
    When $p\geq 1$, we need $\alpha\geq \frac{s-2}{p}$ for $s\geq 2$ and $\alpha\geq 0$ for $s=1$. This concludes the proof of Lemma \ref{SharpExponentExample}.
    \end{proof}
\end{example}
\begin{remark}
    Note that in Theorem \ref{MaxEst}, there is an additional arbitrarily small number $\eps$ in the exponent $\alpha(s)$. Up to this $\eps$ loss, the exponent of $\delta$ is sharp.
\end{remark}
\subsection{Sharpness for the dimension estimate \eqref{DimEstInThm2}}
    Our definition for $E^{s'}$ also gives sharpness for the dimension estimate Corollary \ref{dimest} in the sense that 
    \begin{equation*}
        \dim(E^{s'})=\min \{s'+1,d\}.
    \end{equation*}
\begin{proof}
    When $s'=d-1$ and $d$, $\dim(E_\delta^{s'})=d$ trivially holds since $E^{s'}$ is a subset of $\R^d.$ Therefore, we will prove
    \begin{equation*}
         \dim(E^{s'})\leq s'+1,\quad s'=1,...,d-2.
    \end{equation*}
    This is equivalent to say that for any $\eta>0,$
    \begin{equation}\label{DimUpperBound}
        \dim(E^{s'})\leq s'+1+\eta.
    \end{equation}
    
    By the definition of $E^{s'}$ in \eqref{SharpExampleForDimBound}, it can be seen that 
        $\dim(I_
        1^{s'-1}\times I_2)=s'$, hence for any $\delta>0,$ there exists a countable covering $\{B_i\}$ of $I_
        1^{s'-1}\times I_2$, where $B_i$ is a ball with radius $r(B_i)$ such that 
        \begin{equation}\label{CoverInParameterSpace}
            \sum_i r(B_i)^{s'+\eta}<\delta.
        \end{equation}
        Assume the center of $B_i$ is $(\underline{\mathbf x}^i,r^i)$, 
        it can be verified that 
        \begin{equation*}
           \bigcup_{(\underline{\mathbf x},r)\in B_i} H(\mathbf 0,\underline{\mathbf x},r)\subseteq H(\mathbf 0,\underline{\mathbf x}^i,r^i)+B(\mathbf 0,Cr(B_i))\subseteq \R^d,
        \end{equation*}
        where $C$ depends on the standard moment curve $\gamma$ and dimension. It suffices to cover the right hand side. Note that the right hand side is mortally a $Cr(B_i)$ neighborhood of the moment curve. Therefore, we can use $Cr^i/r(B_i)$ many balls with the same radius $r(B_i)$ to cover it. Denote these balls as $\{B_i^j\}_j^{O(\frac{r^i}{r(B_i)})}.$
        
        Since 
        \begin{equation*}
            E^{s'}\subseteq \bigcup_{i}\bigcup_{(\underline{\mathbf x},r)\in B_i} H(\mathbf 0,\underline{\mathbf x},r)\subseteq \bigcup_i H(\mathbf 0,\underline{\mathbf x}^i,r^i)+B(\mathbf 0,Cr(B_i)),
        \end{equation*}
        $E^{s'}$ can be covered by all these balls $\{B_i^j\}_{i,j}$ where $B_i^j,j=1,...,O(\frac{r^i}{r(B_i)})$ has the same radius $r(B_i)$ and   
        \begin{equation*}
            \sum _{i,j}r(B_i^j)^{s'+\eta+1}=\sum_{i}r(B_i)^{s'+\eta+1}\times \frac{Cr^i}{r(B_i)}\lesssim\sum_i r(B_i)^{s'+\eta}\lesssim\delta<\infty.
        \end{equation*}
    In the computations, we used $r^i\leq 2$ and \eqref{CoverInParameterSpace}. By Remark \ref{HausdorffCotentAndDim}, \eqref{DimUpperBound} holds.
\end{proof}
\begin{remark}
    As mentioned in \cite{HKLO},
    the argument works for general $s'\in (0,d-1].$
\end{remark}

\subsection{Optimal range of $p$}\label{SharpRange}

Now, we start to test other examples that will impose constraints on $p$. In this subsection, we will use our computations in section \ref{sectangency}. For the convenience, we record a special case of equation \eqref{conditions} where $\mathbf x'=\mathbf 0$ and $r'=1$. The conditions become for all $i=2,...,d,$
\begin{equation}\label{0case}
    \Delta_i=x_i(r-1)^{i-1}-x_1^i=0 \mathrm{~and~} t=\frac{x_1}{1-r}\in[-1,1]
\end{equation} 
\begin{example}[Focusing example]\label{FocForHel}
We will prove the following 
\begin{lemma}
    \begin{enumerate}
        \item For maximal estimate \eqref{max1}, $s=d$, the sharp range of $p$ is $p\geq 3$. 
        \item For maximal estimates \eqref{max1}, $s=1,...,d-1$ to be true, $p\geq 2s$.
    \end{enumerate}
\end{lemma}

\begin{proof}

\textit{Item (1).} Choose $f=\chi_{R}$, where $R=H_{\delta}(\mathbf 0,1)\bigcap B(\mathbf 0, (100d)^d\sqrt \delta)$.

For all $r \in [1/2,2]\setminus \{1\}$ 
, the system \eqref{0case} which is a special case of \eqref{conditions} have a solution $\mathbf x_0=(x_1,...,x_d)$. This means for $r\in [5/4,7/4]$, we can find an appropriate moment curve $H(\mathbf x_0,r)$ such that it is tangent to $H(\mathbf 0,1)$. In fact, $\mathbf x_0=\mathbf 0$ and the point of tangency is $(0,...,0)$. By Proposition \ref{exacttangencyvolume}, we know the intersection contains a $\sqrt \delta\times \delta
\times...\times \delta$ rectangle hence 
\begin{equation*}
    \mathcal L^{d-1}(R\bigcap H_\delta(\mathbf 0,r))\sim   \mathcal L^{d-1}(H_\delta(\mathbf 0,1)\bigcap H_\delta(\mathbf 0,r))\sim \delta^{d-\frac{1}{2}}.
\end{equation*}

Note that our radii are in $[5/4,7/4]$ which means $|r-1|\gtrsim 1$. This implies, by \eqref{distanceinparaspace} in the definition, that $dist_p((\mathbf 0,1),(\mathbf x_0,r))\gtrsim 1$.

As a result, we have
\begin{align*}
    \M^1_\delta f(r)&=\sup_{\mathbf x\in\R^d}\frac{1}{\mathcal L^d( H_\delta(\mathbf x,r))}\int_{H_\delta(\mathbf x,r)}|f|\\
    &\gtrsim \frac{1}{\mathcal L^d( H_\delta(\mathbf 0,1))} \delta^{d-\frac{1}{2}} \sim \delta^{d-\frac{1}{2}-(d-1)}=\delta^{\frac{1}{2}}.
\end{align*}
So $\norm{\M^2_\delta f}_p\gtrsim \delta^{\frac{1}{2}}$, while in this case, we have $\norm{f}^p_p= \mathcal L^d(R)\sim \delta^{d-\frac{1}{2}}$.

In order to use $\delta^{-(\frac{d-2}{p}+\eps)}\norm{f}_p=\delta^{(d-\frac{1}{2}-(d-2))/p-\eps}=\delta^{\frac{3}{2p}-\eps}$ to bound $\norm{ \M^1_\delta f}_p\sim \delta^{\frac{1}{2}}$, we need $p\geq 3.$

\noindent\textit{Item (2). } Recall that the maximal function is a function of $(x_{s},...,x_d,r)$. For all $(x_d,r)\in [\eps_1,\eps_2]\times [1-\eta_1,1-\eta_2]$ where $0<\eps_1<\eps_2$ and $\eta_1>\eta_2>0$ to be determined, from the last equation of \eqref{0case}, we have 
\begin{equation*}
    x_1=x_d^{1/d}(r-1)^{(d-1)/d}\leq \eps_2^{1/d}\eta_1^{(d-1)/d}.
\end{equation*}
Once $x_1$ is determined, from other equations, we have for $i=2,3,...,d-1$
\begin{equation*}
    x_i=\frac{x_1^i}{(r-1)^{i-1}}\leq \frac{(\eps_2^{1/d}\eta_1^{(d-1)/d})^i}{\eta_2^{i-1}}\leq \frac{(\eps_2^{1/d}\eta_1^{(d-1)/d})^2}{\eta_2^{d-2}}.
\end{equation*}
We also need 
\begin{equation*}
    t=\frac{x_1}{r-1}\leq \frac{\eps_2^{1/d}\eta_1^{(d-1)/d}}{\eta_2}
\end{equation*}
As long as $\eps_1,\eps_2,\eta_1,\eta_2$ satisfy 
\begin{align*}
    \eps_2^{1/d}\eta_1^{(d-1)/d}&\leq 1/2<1,\\ \frac{(\eps_2^{1/d}\eta_1^{(d-1)/d})^2}{\eta_2^{d-2}}&\leq 1/8<1,\\ \frac{\eps_2^{1/d}\eta_1^{(d-1)/d}}{\eta_2}&\leq (1/8)^{1/(d-2)}<1,
\end{align*}
then for each $(x_d,r)\in [\eps_1,\eps_2]\times [1-\eta_1,1-\eta_2]$, we can find $H(x_1,...,x_d,r)$ such that it is tangent to $H(\mathbf 0,1)$. In fact, if the second inequality holds then the third one must hold. We can choose $\eta_2=\frac{1}{8}$ and $\eta_1=\frac{1}{4}$ and make $\eps_2$ sufficiently small to satisfy the first two inequalities.

Combining this fact with the second part of Proposition \ref{exacttangencyvolume}, we obtain that for each $(x_d,r)\in [\eps_1,\eps_2]\times [1-\eta_1,1-\eta_2]$, and associated $x_i,i=s+1,...,d-1$, as long as $x_i': |x_i'-x_i|\leq \frac{\delta}{(100d)^d}$, then we have 
\begin{equation}\label{ttt}
    \mathcal L^{d}(H_\delta(x_1,...,x_s,x_{s+1}',...,x'_{d-1},x_d,r)\bigcap H_\delta(\mathbf 0,1))\sim \delta^{d-\frac{1}{2}}.
\end{equation}
If we choose $f=\chi_{H_\delta(\mathbf 0,1)}$, then similar computation using \eqref{ttt} implies
\begin{equation*}
    \M_\delta^sf(x_{s+1}',...,x'_{d-1},x_d,r)\sim \delta^{1/2}
\end{equation*}
and 
\begin{equation*}
    \norm{\M_\delta^sf}_p\sim \delta^{1/2}\delta^{\frac{d-1-s}{p}} = \delta^{\frac{d-1-s}{p}+1/2}.
\end{equation*}
In order to use $\norm{f}_p=\delta^{\frac{d-1}{p}}$ to bound $\norm{\M_\delta^s f}_p\sim \delta^{\frac{d-1-s}{p}+1/2}$, we need $p\geq 2s.$
\end{proof}

\end{example}
\begin{remark}
    The example may not be the sharp example for all $s=1,...,d-1$. In particular, for $s=1$, it is conjectured the sharp range is $p> d.$
\end{remark}

\section{Proof of Theorem \ref{MaxEst} when $s=d$}\label{ProjectionArgument} 
In this section, we  will prove the maximal estimate \eqref{max1}, $s=d$.
We first introduce Wolff's seminal result. In \cite{Wolff1997}, he proved the following sharp estimate:
 \begin{theoremA}[Wolff's circular maximal theorem]
     Denote by $C_\delta(x_1,x_2,r)$ the $\delta$- neighborhood of a circle in the plane centered at $(x_1,x_2)\in \R^2$ with radius $r\in [1/2,2]$. Recall Wolff's circular maximal operator is defined as 
     \begin{equation*}
         W_\delta f(r):=\sup_{(x_1,x_2)\in \R^2}\frac{1}{\delta}\int_{C_\delta(x_1,x_2,r)}|f(y)|dy.
     \end{equation*}
     Then for all $\eps>0$ and $p\geq 3$, there is a constant $C_\eps$ such that 
     \begin{equation}\label{WolffCircular3to3}
         \norm{W_\delta f}_{L^p([1/2,2])}\leq C_\eps\delta^{-\eps}\norm{f}_{L^p(\R^3)}.
     \end{equation}
 \end{theoremA}
 In our case, the two dimensional moment curve is parabola which means this result for circles cannot be applied directly. 
 
 However, combining with Schlag's result \cite{Schlag2003}, Zahl \cite{Zahl2012} generalizes this result to the family of curves satisfying the so-called Sogge's cinematic curvature condition. Additionally, in \cite{pyz} and \cite{zahl2025maximalfunctionsassociatedfamilies}, the authors prove the same sharp estimate under a condition that may be slightly rigorously weaker than Sogge's original definition. In particular, in the appendix of \cite{zahl2025maximalfunctionsassociatedfamilies}, the author verifies that parabolas parametrized by vertex $(x_1,x_2)$ and dilation $r$, are indeed a family of curves that satisfies the cinematic curvature condition. So we  will assume the above Theorem A holds with circle $C$ replaced by parabola $P$. Similar to the circle, We also denote as $P_\delta(x_1,_2,r)$ the $\delta$-neighborhood of parabola $P(x_1,x_2,r)$ with vertex $(x_1,x_2)$ and dilate factor $r$.
 
By applying this planar sharp result, we can prove Theorem \ref{MaxEst} for $s=d$.
\begin{proof}[Proof of the boundedness of $\M_\delta^d$]
Assume $f\geq 0$ and denote $\mathbf x=(x_1,...x_d)=(\overline{\mathbf x}_2,\underline{\mathbf x}_{d-2})=(\overline{\mathbf x},\underline{\mathbf x}).$ If we write 
\begin{equation*}
      \M_\delta^1f(r)=\sup_{\mathbf x\in \R^d}\frac{1}{\mathcal L^d(H_\delta(\mathbf x,r))}\int_{H_\delta(\mathbf x,r)}f(\mathbf s)d\mathbf s=\sup_{\mathbf x\in \R^d}\mathcal A_\delta f(\mathbf x,r),
\end{equation*}
where $\mathbf s=(s_1,...,s_d)=(\overline{\mathbf s},\underline{\mathbf s})$ and
\begin{equation}\label{RoughAveragingOperator}
    \mathcal A_\delta f(\mathbf x,r)=\frac{1}{\mathcal L^d(H_\delta(\mathbf x,r))}\int_{H_\delta(\mathbf x,r)}f
\end{equation} is the averaging operator associated to moment curves. Using Fubini's theorem, we express $\mathcal A_\delta $ as an iterated integral involving the projection of $H_\delta(\mathbf x,r)$ in the first two coordinates $\overline{\mathbf x}_2=\overline{\mathbf x}$. Recall that in \eqref{p=pi}, we proved that

    \begin{equation*}
        P_\delta = \pi_2(H_\delta(\mathbf x,r)),
    \end{equation*}
    where 
    \begin{equation*}
P(x_1,x_2,r)=\{(x_1,x_2)+r(t,t^2):t\in[-1,1]\}.
     \end{equation*}
    In the following computations, $C$ is a constant that depends only on dimension and may change from line to line. For any fixed $\mathbf x,$
\begin{align*}
    \mathcal A_\delta f(\mathbf x,r)&\leq \frac{C}{\delta^{d-1}}\int_{H_\delta(\mathbf x,r)}f(\mathbf s)d\mathbf s\\
    &=\frac{C}{\delta}\int_{P_\delta(\overline{\mathbf x},r)}\left( \frac{1}{\delta^{d-2}}\int_{\pi_2^{-1}(\overline{\mathbf s})\cap H_\delta(\mathbf x,r)} f(\mathbf s) d\underline{\mathbf s}\right)d\overline{\mathbf s}\\
    &=\frac{C}{\delta}\int_{P_\delta(\overline{\mathbf x},r)}\left( \frac{1}{\delta^{d-2}}\int_{\pi_2^{-1}(\overline{\mathbf s})\cap H_\delta(\mathbf x,r)} F_{\overline{\mathbf s}}(\underline{\mathbf s}) d\underline{\mathbf s}\right)d\overline{\mathbf s},
\end{align*}
where $F_{\overline{\mathbf s}}(\underline{\mathbf s}):=f(\overline{\mathbf s},\underline{\mathbf s})=f(\mathbf s)$. We estimate the inner integral via H\"older's inequality. Note that for any fixed $\overline{\mathbf s},\pi_2^{-1}(\overline{\mathbf s})\cap H_\delta(\mathbf x,r)$ is a $\delta$-ball in $\R^{d-2}$. So 
\begin{align*}
     \mathcal A_\delta f(\mathbf x,r)&\leq \frac{C}{\delta}\int_{C_\delta(\overline{\mathbf x},r)}\left( \norm{F_{\overline{\mathbf s}}}_{L^p(\R^{d-2})}\delta^{-\frac{d-2}{p}} \right)d\overline{\mathbf s}\\
     &=\frac{C}{\delta}\int_{C_\delta(\overline{\mathbf x},r)} G(\overline{\mathbf s}) \delta^{-\frac{d-2}{p}} d\overline{\mathbf s},
\end{align*}
where $G(\overline{\mathbf s}):=\norm{F_{\overline{\mathbf s}}}_{L^p(\R^{d-2})}.$
Take supremum over $\overline{\mathbf x}$ in the right hand side,
\begin{align*}
      \mathcal A_\delta f(\mathbf x,r)\leq C\delta^{-\frac{d-2}{p}}W_\delta G(r),
\end{align*}
Since the right hand side is a uniform upper bound independent of $\mathbf x$, take supremum over $\mathbf x$, 
\begin{equation*}
     \M_\delta^df(r)\leq C\delta^{-\frac{d-2}{p}}W_\delta G(r).
\end{equation*}
Then applying the parabola version estimate \eqref{WolffCircular3to3}, we obtain 
\begin{align*}
     \norm{\M^d_\delta f}_{L^p([1/2,2])}&\leq C\delta^{-\frac{d-2}{p}}\norm{W_\delta G}_{L^p([1/2,2])}\\
     &\leq C_\eps\delta^{-(\frac{d-2}{p}+\eps)}\norm{G}_{L^p(\R^2)}\\
     &\leq C_\eps\delta^{-(\frac{d-2}{p}+\eps)}\norm{f}_{L^p(\R^d)}
\end{align*}
which concludes the proof. 
\end{proof}
\begin{remark}
    Note that the above argument can be used to tackle any family of spatial curves whose projection forms a family of planar cinematic curves.  
\end{remark}

\section{Proof of Theorem \ref{MaxEst} when $s=1,...,d-1$} \label{SobolevEmbedding}
In \cite{Ko_Lee_Oh_2023}, the authors Ko, Lee and Oh have established a local smoothing estimate for maximal averages associated with non-degenerate curves in $\R^d$. The connection between maximal estimate and local smoothing estimate is well-known in the harmonic analysis literature. We record the relevant result from \cite{Ko_Lee_Oh_2023} in Theorem B below and use it to derive the $L^p$ bound \eqref{max1} for $s=1,...,d-1$.



Each of the following subsections describes a step of the proof of Theorem \ref{MaxEst} for $s=1,...,d-1$.
By replacing $f$ by $|f|$, we can always assume $f\geq0.$ Since the space of compactly supported functions, $C^\infty_c(\R^d)$ is a dense subspace of $L^p$ space. We also assume $f$ is compactly supported.

\subsection{Reduction to a smooth averaging operator}

As a first step, we bound $\M_\delta^s$ by a smooth version of it. Let us introduce the smooth cut-off functions we need. Their existence is justified in the appendix \ref{existenceofcutoff}. Denote $\tilde{\mathbf t}=(t_2,...,t_d)$ and assume $\psi(\tilde{\boldsymbol{t}}):\R^{d-1}\to \R$ is a nonnegative smooth function that satisfies that 
\begin{enumerate}
    \item $\hat{\psi}$ is compactly supported and nonnegative.
    \item $\int_{\R^{d}} \psi>0$.
    \item $\psi(\tilde{\boldsymbol{t}})\geq 1$ for $\tilde{\boldsymbol{t}}\in B(\mathbf 0,C_d)$ for some $C_d>0$ to be determined later.
\end{enumerate}
As conventions, $\psi_\delta(\cdot)=\frac{1}{\delta^{d-1}}\psi(\frac{\cdot}{\delta})$ and one can compute that $\int_{\R^{d-1}} \psi_\delta=\int_{\R^{d-1}} \psi=c$.
This function is used to tackle the $\delta$-thickening in physical space and smoothly cut off the frequency space.

Assume $\chi_1(r):\R\to \R$ is a non-negative function that is
compactly supported in $[1/4,3]$ and $\chi_1(r)=1 $ for $r\in[1/2,2].$
This function restricts the radius to the finite interval $[1/2,2]$.

Assume $\chi_0(u):\R\to \R$ is a non-negative function that is
compactly supported in $[-2,2]$ and $\chi_0(u)=1 $ for $u\in [-1.5,1.5]$. This function smoothly cuts off the moment curves. With these notations, we define the following smoothed-out averaging operator associated to moment curves, 
\begin{align*}
\mathfrak{A}_\delta f(&\mathbf x,r)\\:=
    &\chi_1(r)\int_{\R^d}f( x_1-ru,x_2-ru^2-t_2,...,x_d-ru^d-t_d) \psi_{\delta}(\tilde{\mathbf t})d\tilde{\mathbf t}\chi_0(u)du\\
    =&\chi_1(r)\int_{\R^d}f(\mathbf x-r\gamma(u)-(0,\tilde{\boldsymbol{t}}))\psi_{\delta}(\tilde{\boldsymbol{t}})d\tilde{\boldsymbol{t}}\chi_0(u)du.
\end{align*}
As we mentioned before, $(t_2,...,t_d)$ and $\psi_\delta$ are used to take the average over the $\delta$-neighborhood of the moment curves. $\chi_0$ is used to limit the parameter of the moment curve roughly in $[-1,1]$. For $r\in [1/2,2]$, $\chi_1$ makes the average a smooth function of radius $r$. 

From the definition, we trivially have
\begin{align*}
   1 _{[1/2,2]}(r)&\leq \chi_1(r)\\
   1_{[-1,1]}(u)&\leq \chi_0(u),
\end{align*}
where $1_E$ is the characteristic function of set $E$.
Combining these facts with our carefully chosen $\psi$, recalling the definition of averaging operator \eqref{RoughAveragingOperator} we claim \begin{equation}\label{smoothdomination}
\begin{aligned}
     \mathcal A_\delta f(\mathbf x,r)&\leq  \chi_1(r)\int_{\R^d}f(\mathbf x-r\gamma(u)-(0,\tilde{\boldsymbol{t}}))\psi_{\delta}(\tilde{\boldsymbol{t}})d\tilde{\boldsymbol{t}}\chi_0(u)du\\
     &= \mathfrak{A}_\delta  f(\mathbf x,r).
\end{aligned}
\end{equation}
We include a rigorous proof for \eqref{smoothdomination} in appendix \ref{smoothversiondomination}. 
The smoothed-out averaging operator on the right hand side is more amenable to Fourier analytic tools. 
Slightly abusing the notations, we continue to denote our the maximal operator as $\M_\delta^s $, i.e. 
\begin{equation*}
\begin{aligned}
      \M_\delta^s f(\underline{\mathbf x},r)&=\sup_{\overline{\mathbf x}\in I_1^s}   \mathfrak{A}_\delta  f(\overline{\mathbf x},\underline{\mathbf x},r)\\
      &=\sup_{\overline{\mathbf x}\in I_1^s}  \chi_1(r)\int_{\R^d}f(\mathbf x-r\gamma(u)-(0,\tilde{\boldsymbol{t}}))\psi_{\delta}(\tilde{\boldsymbol{t}})d\tilde{\boldsymbol{t}}\chi_0(u)du.
\end{aligned}
\end{equation*}
We will prove the maximal estimates \eqref{max1} for $s=1,...,d-1$ with this new definition, which will lead to Theorem 2 in view of \eqref{smoothdomination}.

\subsection{Decomposition of $\mathfrak A_\delta$ through its multiplier}Denote $
\boldsymbol{\xi}=(\xi_1, \tilde{\boldsymbol{\xi}})=(\xi_1,...,\xi_d)$. Since $\mathfrak A_\delta$ is a convolution operator, let us compute the multiplier for its Fourier transform. We want to remind the reader that typically, the Fourier transform is with respect to the space variables, i.e. $\mathbf x$ and the variable in dual frequency space is denoted as $\boldsymbol{\xi}.$ Assume $f\in C_c^\infty(\R^d)$,
    \begin{equation*}
    \begin{aligned}
           \mathfrak{A}_\delta  f(\mathbf x,r)&=\chi_1(r)\int_{\R^d}f(\mathbf x-r\gamma(u)-(0,\tilde{\boldsymbol{t}}))\psi_{\delta}(\tilde{\boldsymbol{t}})d\tilde{\boldsymbol{t}}\chi_0(u)du\\
           &=\chi_1(r)\int_{\R^d}\int_{\R^d}\hat{f}(\boldsymbol{\xi})e^{2\pi i (\mathbf x-r\gamma(u)-(0,\tilde{\boldsymbol{t}}))\cdot \boldsymbol{\xi}}d\boldsymbol{\xi} \psi_{\delta}(\tilde{\boldsymbol{t}})d\tilde{\boldsymbol{t}}\chi_0(u)du.
    \end{aligned}
\end{equation*}
 $\int\psi_\delta=c$ and $\chi_0(u)$ is compactly supported. So the above integral is absolutely convergent. Applying Fubini's theorem, we interchange the order of the integration, 
\begin{equation*}
\begin{aligned}
     \mathfrak{A}_\delta  f(\mathbf x,r) &= \chi_1(r)\int_{\R^d}\int_{\R^d}\hat{f}(\boldsymbol{\xi})e^{2\pi i (\mathbf x-r\gamma(u)-(0,\tilde{\boldsymbol{t}}))\cdot \boldsymbol{\xi}}d\boldsymbol{\xi} \psi_{\delta}(\tilde{\boldsymbol{t}})d\tilde{\boldsymbol{t}}\chi_0(u)du\\
     =\chi_1(r)&\int_{\R^d}\hat{f}(\boldsymbol{\xi})\left(\int_\R\int_{\R^{d-1}}e^{-2\pi i (r\gamma(u)+(0,\tilde{\boldsymbol{t}}))\cdot \boldsymbol{\xi}} \psi_{\delta}(\tilde{\boldsymbol{t}})d\tilde{\boldsymbol{t}}\chi_0(u)du\right)e^{2\pi i\mathbf x \cdot \boldsymbol{\xi}}d\boldsymbol{\xi}.
\end{aligned}
\end{equation*}
Using the product structure of the inner integral and the definition of the Fourier transform, we obtain
\begin{equation*}
    \begin{aligned}
         \mathfrak{A}_\delta  &f(\mathbf x,r) \\
         =&\chi_1(r)\int_{\R^d}\hat{f}(\boldsymbol{\xi})\left(\int_\R\int_{\R^{d-1}}e^{-2\pi i (r\gamma(u)+(0,\tilde{\boldsymbol{t}}))\cdot \boldsymbol{\xi}} \psi_{\delta}(\tilde{\boldsymbol{t}})d\tilde{\boldsymbol{t}}\chi_0(u)du\right)e^{2\pi i\mathbf x \cdot \boldsymbol{\xi}}d\boldsymbol{\xi}\\
                      =&\chi_1(r)\int_{\R^d}\hat{f}(\boldsymbol{\xi})\int_\R e^{-2\pi r\boldsymbol{\xi}\cdot\gamma(u)}\left(\int_{\R^{d-1}} e^{-2\pi i \tilde{\boldsymbol{t}}\cdot \tilde{\boldsymbol{\xi}}} \psi_{\delta}(\tilde{\boldsymbol{t}})d\tilde{\boldsymbol{t}}\right)\chi_0(u)due^{2\pi i\mathbf x \cdot \boldsymbol{\xi}}d\boldsymbol{\xi}\\
           =&\chi_1(r)\int_{\R^d} \widehat{f}(\boldsymbol{\xi})     
           \left(\int_\R e^{-2\pi ir\boldsymbol{\xi}\cdot \gamma(u)}\widehat{\psi}(\delta\tilde{\boldsymbol{\xi}})\chi_0(u)du \right)e^{2\pi i\boldsymbol{\xi}\cdot \mathbf x}d\boldsymbol\xi.
    \end{aligned}
\end{equation*}
So we obtain the symbol of the multiplier,
\begin{equation*}
   m_{\delta,r}(\boldsymbol{\xi}):= \chi_1(r)\widehat{\psi}(\delta\tilde{\boldsymbol{\xi}})\int_\R e^{-2\pi ir\boldsymbol{\xi}\cdot \gamma(u)}\chi_0(u)du.
\end{equation*}

By decomposing the multiplier $m_{\delta,r}$ into different pieces, we divide the operator into three parts. Assume $\mathfrak{a}_d(\boldsymbol{\xi})$ is a smooth nonnegative cut-off function with compact support $\{\boldsymbol{\xi}\in \R^d: |\boldsymbol{\xi}|\leq2\}$ such that $\mathfrak{a}_d(\boldsymbol{\xi})=1$ on $\{\boldsymbol{\xi}\in \R^d: |\boldsymbol{\xi}|\leq 1\}$ 
. Define 
\begin{equation*}
    \phi(\boldsymbol{\xi}):=\mathfrak{a}_1\left(\frac{\xi_1}{(100d)^d|\sum_{i=2}^d\xi_i^2|^{1/2}}\right)=\mathfrak{a}_1\left(\frac{\xi_1}{(100d)^d|\tilde{\boldsymbol{\xi}}|}\right).
\end{equation*}
Note that similar to $\tilde{\boldsymbol{t}},$ $\tilde{\boldsymbol{\xi}}$ represents the last $d-1$ coordinates of $\boldsymbol{\xi}$. Then $\phi$ is a smooth function on $\{\boldsymbol{\xi}\in \R^d: |\boldsymbol{\xi}|\geq 1\}$ such that $\phi(\boldsymbol{\xi})=1$ on $|\xi_1|\leq  (100d)^d |\tilde{\boldsymbol{\xi}}|$ and $\phi(\boldsymbol{\xi})=0$ on $|\xi_1|>  2(100d)^d |\tilde{\boldsymbol{\xi}}|$. Then we decompose the symbol as follows:

\begin{equation*}
     m_{\delta,r}(\boldsymbol{\xi})=m_{\delta,r}^{low}(\boldsymbol{\xi})+m_{\delta,r}^{high0}(\boldsymbol{\xi})+m_{\delta,r}^{high1}(\boldsymbol{\xi}),
\end{equation*}
where the first part is
\begin{equation}\label{lowpartmulti}
   m_{\delta,r}^{low}(\boldsymbol{\xi}):=\chi_1(r)     
           \int_\R e^{-2\pi ir\boldsymbol{\xi}\cdot \gamma(u)}\widehat{\psi}(\delta\tilde{\boldsymbol{\xi}})\chi_0(u)\mathfrak{a}_d(\boldsymbol{\xi})du .
\end{equation}
This is the low frequency part of the multiplier. The second part is 
\begin{equation}\label{high0}
    m_{\delta,r}^{high0}(\boldsymbol{\xi}):=\chi_1(r)     
           \int_\R e^{-2\pi ir\boldsymbol{\xi}\cdot \gamma(u)}\widehat{\psi}(\delta\tilde{\boldsymbol{\xi}})\chi_0(u)(1-\mathfrak{a}_d(\boldsymbol{\xi}))(1-\phi(\boldsymbol{\xi}))du 
\end{equation}
This is the high frequency part of the multiplier whose support is contained in $|\xi_1|\geq  2(100d)^d |\tilde{\boldsymbol{\xi}}|$. The third part is 
\begin{equation}\label{high1}
    m_{\delta,r}^{high1}(\boldsymbol{\xi}):=\chi_1(r)     
           \int_\R e^{-2\pi ir\boldsymbol{\xi}\cdot \gamma(u)}\widehat{\psi}(\delta\tilde{\boldsymbol{\xi}})\chi_0(u)(1-\mathfrak{a}_d(\boldsymbol{\xi}))\phi(\boldsymbol{\xi})du.
\end{equation}
This is the high frequency part of the symbol whose support is contained in $|\xi_1|\leq  2(100d)^d |\tilde{\boldsymbol{\xi}}|$. Therefore, the decomposition of multiplier will induce a decomposition of the averaging operator i.e. 
\begin{equation*}
     \mathfrak{A}_\delta  f=    \mathfrak{A}^{low}_\delta  f+   \mathfrak{A}^{high0}_\delta  f+   \mathfrak{A}^{high1}_\delta  f,
\end{equation*}
and a corresponding decomposition of maximal operator,
\begin{equation*}
    \M_\delta^s f=\M_\delta^{s,low} f+\M_\delta^{s,high0} f+\M_\delta^{s,high1} f.
\end{equation*}

Before we move on, we briefly explain the motivation of this decomposition. We first decompose the frequency space into low frequency part, i.e. $\{\boldsymbol{\xi}\in\R^d: |\boldsymbol{\xi}|\leq 1\}$ and high frequency part $\{\boldsymbol{\xi}\in\R^d: |\boldsymbol{\xi}|> 1\}$. This is because, by the uncertainty principle, the low frequency part is easier to handle. Heuristically, the uncertainty principle says that if $\widehat{f}$ is supported in some ball with radius $R$, then $|f|$ is essentially constant in scale $\frac{1}{R}$. Due the the compactness of $\mathfrak{a}_d,$ it is easy to verify that if we fix $(\underline{\mathbf x},r)=(x_3,...,x_d,r)$, the Fourier transform of $\mathfrak{A}^{low}_\delta f(\mathbf x,r)$ with respect to $(x_1,x_2)$ is contained in a compact set. This mean that $  \mathfrak{A}^{low}_\delta f(\mathbf x,r)$ (as a function of $(x_1,x_2)$) is essentially constant in a compact set. So intuitively there is no difference between $  \mathfrak{A}^{low}_\delta f(\mathbf x,r)$ and $ \M_\delta^{s,low} f(\underline{\mathbf x},r)$, while  the former averaging operator can be nicely controlled. We make this intuition precise in Subsection \ref{estforlowpart}.

The high frequency part is more complicated. Since 
\begin{equation*}
   m_{\delta,r}(\boldsymbol{\xi})= \chi_1(r)\widehat{\psi}(\delta\tilde{\boldsymbol{\xi}})\int_\R e^{-2\pi ir\boldsymbol{\xi}\cdot \gamma(u)}\chi_0(u)du
\end{equation*}
is also an oscillatory integral, its behavior is dictated by the location and nature of the critical points of its phase function $\boldsymbol{\xi}\cdot\gamma(u).$ We note that 
\begin{equation*}
    \frac{\p (\boldsymbol{\xi}\cdot\gamma(u))}{\p u}=\boldsymbol{\xi}\cdot\gamma'(u)= \xi_1+\sum_{i=2}^d du^{d-1}\xi_i. 
    \end{equation*}
    Since $u\in [-1.5,1.5], $ as long as $|\xi_1|\geq (100d)^d|\tilde{\boldsymbol{\xi}}|$, then 
    \begin{equation*}
    \begin{aligned}
        \left|\frac{\p (\boldsymbol{\xi}\cdot\gamma(u))}{\p u}\right|&=\left|\xi_1+\sum_{i=2}^d du^{d-1}\xi_i\right|\\
        &\geq |\xi_1|- \left(\sum_{i=2}^d du^{d-1}|\xi_i|\right)\\
        &\gtrsim |\xi_1|\\
        &\gtrsim 1,
    \end{aligned}
    \end{equation*}
    when $|\boldsymbol{\xi}|\geq 1$. 
    Thus, the phase function of the oscillatory integral $m_{\delta,r}^{high0}$ has no critical point, allowing one to integrate by parts for arbitrarily many times, gaining a super-polynomial decay in $|\boldsymbol{\xi}|$.
    Hence its contribution is negligible. This argument has be made rigorous in Subsection \ref{secintbyparts}.

    Treatment of $m_{\delta,r}^{high1}$ and the resulting maximal operator $\M_\delta^{s,high1} $ is the crux of the problem. This is where we need to appeal to \cite{Ko_Lee_Oh_2023}. This part of the proof is presented in Subsection \ref{seclocalsmoothing}.


\subsection{Estimating $\M_\delta^{s,low} f=\sup \mathfrak{A}_\delta^{low}f$}\label{estforlowpart} We first tackle the low frequency part.  We claim
    \begin{proposition}\label{1part}For any $p\geq 1$, there is some constant $C$ depending only on dimension and $p$ (not depending on $\delta$) such that 
    \begin{equation}\label{lowpart}
           \norm{\sup_{\overline{\mathbf x}\in \R^s}   \mathfrak{A}^{low}_\delta  f(\mathbf x,r)}_{L^p(\R^{s'})}\leq C\norm{f}_{L^p(\R^{d})}.
    \end{equation}      
    \end{proposition}
    This is definitely stronger than 
    \begin{equation*}
        \norm{\M_\delta^{s,low} f}_p= \norm{\sup_{\overline{\mathbf x}\in I_1^s}   \mathfrak{A}^{low}_\delta  f(\mathbf x,r)}_{L^p(\R^{s'})}\leq C_\eps \delta^{-\eps}\norm f_p,
    \end{equation*}
    in the sense that the range of $\overline{\mathbf x}$ over which we are taking supremum is the entire $\R^s$ and the finite boundedness instead of $\eps$-loss estimate holds.
\begin{proof}
    Let us recall the definition of $ \mathfrak{A}^{low}_\delta  f$ from \eqref{lowpartmulti}
\begin{align*}
     \mathfrak{A}^{low}_\delta  f(\mathbf x,r)=\chi_1(r)\int_{\R^d} \widehat{f}(\boldsymbol{\xi})     
           \left(\int_\R e^{-2\pi ir\boldsymbol{\xi}\cdot \gamma(u)}\widehat{\psi}(\delta\tilde{\boldsymbol{\xi}})\chi_0(u)\mathfrak{a}_d(\boldsymbol{\xi})du \right)e^{2\pi i\boldsymbol{\xi}\cdot \mathbf x}d\boldsymbol\xi.
\end{align*}
Due to the cut-off function $\mathfrak{a}_d$, it is easy to check that for any fixed $(\underline{\mathbf x},r)=(x_{s+1},...,x_d,r),$ 
     \begin{equation}\label{FT2}
         \mathrm{supp}\left(\mathcal F_{\overline{\mathbf x}}( \mathfrak{A}^{low}_\delta  f(\cdot,\underline{\mathbf x},r))\right)\subseteq B^s(\mathbf 0,1),
     \end{equation}
     where $B^s(\mathbf 0,1)$ is the ball in $\R^s$ centered at zero with radius $1$ and $\mathcal F_{\overline{\mathbf x}}(\cdot)$ is the Fourier transform with respect to $\overline{\mathbf x}$. Indeed, 
     \begin{equation}\label{compactsupport1}
           \begin{aligned}
         \mathcal F&_{\overline{\mathbf x}}( \mathfrak{A}^{low}_\delta  f(\cdot,\underline{\mathbf x},r))\\=
         &\chi_1(r)\int_{\R^s} e^{-2\pi i\overline{\boldsymbol{\eta}}\cdot \overline{\mathbf x}}\int_{\R^d}\widehat{f}(\boldsymbol{\xi})  \left(\int_\R e^{-2\pi ir\boldsymbol{\xi}\cdot \gamma(u)}\chi_0(u)du \right)\widehat{\psi}(\delta\tilde{\boldsymbol{\xi}}))\mathfrak{a}_d(\boldsymbol{\xi})e^{2\pi i\boldsymbol{\xi}\cdot \mathbf x}d\boldsymbol\xi d\overline{\mathbf x}\\
         =&\chi_1(r)\int_{\R^d}\widehat{f}(\boldsymbol{\xi}) \left(\int_\R e^{-2\pi ir(\overline{\boldsymbol{\eta}},\underline{\boldsymbol{\xi}})\cdot \gamma(u)}\chi_0(u)du \right)\widehat{\psi}(\delta\eta_2,...,\delta\eta_s,\delta\underline{\boldsymbol{\xi}})\mathfrak{a}_d(\overline{\boldsymbol{\eta}},\underline{\boldsymbol{\xi}})e^{2\pi i\underline{\boldsymbol{\xi}}\cdot\underline{\mathbf{x}}}d\underline{\mathbf x}.
     \end{aligned} 
     \end{equation}
     Note that $\mathfrak{a}_d$ vanishes for $\overline{\boldsymbol{\eta}}$ out side of $B^s(\mathbf 0,1)$. Now we are going to apply the uncertainty principle. One specific form of the principle is the following Bernstein's inequality or reverse H\"older's inequality.   
    \begin{proposition}\cite[item (2), Proposition 5.3]{WolffBook} Suppose $F\in L^1+L^2$ and supp($\widehat{F}$)$\subseteq B(\mathbf 0,R)$. Then for any $1\leq p\leq q\leq \infty$,
    \begin{equation}\label{bernstein}
        \norm{F}_q\leq CR^{d(\frac{1}{p}-\frac{1}{q})}\norm{F}_p.
    \end{equation}
    \end{proposition}
\begin{remark}
    The reason why Bernstein's inequality is also called reverse H\"older's inequality is, for any function $f$ supported in a ball $B_R$ with radius $R$ and its averaging $L^p$ norm
    \begin{equation*}
        \norm{f}_{L^p_\sharp(B_R)}:=\left(\frac{1}{\mathcal L^d(B_R)}\int_{B_R} |f|^p \right)^\frac{1}{p},
    \end{equation*}
    It is elementary to see by H\"older's inequality that if $1\leq p\leq q\leq \infty,$
    \begin{equation*}
         \norm{f}_{L^p_\sharp(B_R)}\lesssim  \norm{f}_{L^q_\sharp(B_R)}\text{ which is }\norm{f}_p\lesssim R^{d(\frac{1}{p}-\frac{1}{q})}\norm{f}_q.
    \end{equation*}
    I.e. we can use the (averaging) norm of $f$ with higher exponent $q$  to bound the norm of $f$ with lower exponent $p$, while Bernstein's inequality tells us if Fourier support information is known, we can use lower exponent norm to control higher exponent norm, that is equation \eqref{bernstein}. Of course, the right hand side should also be compensated a factor of $R^{d(\frac{1}{p}-\frac{1}{q})}$. 
    
    This is exactly what we will repeatedly apply in the rest of the section. We will use an ordinary $L^p$ norm, $p<\infty$ to bound the $L^\infty$ norm, i.e. the supremum, at an affordable cost.
\end{remark}
     We apply inequality \eqref{bernstein} with $d=s$ and 
     \begin{equation*}
         F_{\underline{\mathbf{x}},r}(\overline{\mathbf x})= \mathfrak{A}^{low}_\delta  f(\overline{\mathbf x},\underline{\mathbf x},r),
     \end{equation*}
      for any $p\geq 1,$
     \begin{equation}\label{sup<p}
        \sup_{\overline{\mathbf x}\in \R^s}   \mathfrak{A}^{low}_\delta  f(\mathbf x,r)\lesssim \norm{ \mathfrak{A}^{low}_\delta  f(\mathbf x,r)}_{L^p_{\overline{\mathbf x}}(\R^s)}. 
     \end{equation}
     Raising both sides by $p$-th power and integrating over $(\underline{\mathbf x},r)$, in order to prove \eqref{lowpart}, it suffices to prove that 
     \begin{equation}\label{localsmoothingtrivial1}
          \norm{\norm{ \mathfrak{A}^{low}_\delta  f(\mathbf x,r)}_{L^p_{\overline{\mathbf x}}(\R^s)}}_{L_{\underline{\mathbf x},r}^p(\R^{s'})}= \norm{   \mathfrak{A}^{low}_\delta  f(\mathbf x,r)}_{L^p_{{\mathbf x},r}(\R^{d+1})}\lesssim\norm{f}_{L^p(\R^d)}.
     \end{equation}
     According to \cite[Proposition 1.3]{WolffBook}, it is easy to see that 
     \begin{equation*}
         m_\delta (\boldsymbol{\xi}):=\int_\R e^{-2\pi ir\boldsymbol{\xi}\cdot \gamma(u)}\widehat{\psi}(\delta\tilde{\boldsymbol{\xi}})\chi_0(u)\mathfrak{a}_d(\boldsymbol{\xi})du \in  C_c^\infty(\R^d),
     \end{equation*}   
      then \eqref{localsmoothingtrivial1} is deduced directly from Young's convolution inequality and the fact $\int \psi_\delta\sim1$. Indeed, since $\chi_1(r)$ is a compactly supported function,  
     \begin{equation}\label{dropr}
     \begin{aligned}
          \norm{  \mathfrak{A}^{low}_\delta f(\mathbf x,r)}_{L^p_{{\mathbf x},r}(\R^{d+1})}&=\norm{\norm{\chi_1\widecheck{m^{}_{\delta}}*f}_{L^p_r(\R)}}_{L_\mathbf{x}^p(\R^{d})} \\
          &\lesssim\norm{\widecheck{m^{}_{\delta}}*f}_{L_\mathbf{x}^p(\R^{d})}\lesssim\norm{\widecheck{m_{\delta}^{}}}_{L^1_\mathbf x(\R^d)}\norm{f}_p.
     \end{aligned}
     \end{equation}
     It suffices to prove $\norm{\widecheck{m_{\delta}^{}}}_{L^1(\R^d)}$ is bounded. To simplify the computation, we write $  m_\delta$ into the following form. Recall that $\mathfrak{a}_1=1$ on $[-1,1],$ so we can multiply $m_\delta$ by $\mathfrak{a}_1$ without affecting its value, that is 
     \begin{align*}
          m_\delta (\boldsymbol{\xi})&=\left(\widehat{\psi}(\delta\tilde{\boldsymbol{\xi}})\mathfrak{a}_1(\frac{\xi_1}{2})\right)\times \left(\mathfrak{a}_d(\boldsymbol{\xi}) \int_\R e^{-2\pi ir\boldsymbol{\xi}\cdot \gamma(u)}\chi_0(u)du\right)\\
          &=m_1(\boldsymbol{\xi})\times m_2(\boldsymbol{\xi}).
     \end{align*}
     Therefore, we have 
     \begin{equation*}
     \begin{aligned}
           \widecheck{m_{\delta}^{}}(\mathbf x)&=\mathcal F^{-1} (m_1) * \mathcal F^{-1}(m_2)(\mathbf x)\\
           &=(2\widecheck{\mathfrak{a}_1}(2\cdot)\psi_\delta*\int \widecheck{\mathfrak{a}_d}(\cdot-r\gamma(u))\chi_0(u)du)(\mathbf x)\\
           &=\int 2\widecheck{\mathfrak{a}_1}(2\eta_1)\psi_\delta(\tilde{\boldsymbol{\eta}})\int \widecheck{\mathfrak{a}_d}(\mathbf x-\boldsymbol{\eta}-r\gamma(u))\chi_0(u)dud\boldsymbol{\eta}.
     \end{aligned}
     \end{equation*}

     So it suffices to prove 
     \begin{equation}\label{mlow}
       \norm{\widecheck{m_1}*\widecheck{m_2}}_{L^1(\R^d)}  =\norm{\widecheck{\mathfrak{a}_1}(2\cdot)\psi_\delta*\int \widecheck{\mathfrak{a}}(\cdot-r\gamma(u))\chi_0(u)du}_{L^1(\R^d)}\lesssim 1.
     \end{equation}
     Since we have 
     \begin{equation*}
     \begin{aligned}
         \norm{\widecheck{m_1}}_{L^1(\R^d)}= \norm{\widecheck{\mathfrak{a}_1}(2\cdot)\psi_\delta}_{L^1(\R^d)}&=\int \widecheck{\mathfrak{a}_1}(2x_1)\psi_\delta(\tilde{\mathbf x})dx_1d\tilde{\mathbf x}\\&=\int \widecheck{\mathfrak{a}_1}(2x_1)dx_1\int \psi_\delta(\tilde{\mathbf x})d\tilde{\mathbf x}\lesssim 1
     \end{aligned}
     \end{equation*}
     due to  $\int \psi_\delta\sim 1$ and $\mathfrak{a}_1\in \mathcal S(\R)$ and 
     \begin{equation*}
         \norm{\int \widecheck{\mathfrak{a}}(\cdot-r\gamma(u))\chi_0(u)du}_{L^1(\R^d)}\leq \norm{\widecheck{\mathfrak{a}_1}}_{L^1(\R^d)}\norm{\chi_0}_{L^1(\R^d)}\lesssim1
     \end{equation*}
     due to the Fubini-Tonelli and $\widecheck{\mathfrak{a}}\in\mathcal S(\R^d).$ Applying Young's inequality to \eqref{mlow} and plugging these two estimates, we obtain the desired result.
 \end{proof}
\begin{remark}
       For \eqref{sup<p}, while in the case where the $\sup$ or $L^\infty$ norm is considered over a compact set or a half plane, say $(0,\infty)$, Bernstein's inequality will be replaced by a standard application of fundamental theorem of calculus, as in the proof of Sobolev embedding inequality.
\end{remark}

 \subsection{Estimating $\M_\delta^{s,high0} f=\sup \mathfrak{A}_\delta^{high0}f$}\label{secintbyparts} Now we start to tackle the high frequency part where there is no critical point of the phase function. We claim 
  \begin{proposition}\label{2part}For any $p\geq 1$, there is some constant $C$ depending only on dimension and $p$ (not depending $\delta$) such that 
        \begin{equation*}
            \norm{ \M_\delta^{s,high0} f(\underline{\mathbf x},r))}_{L^p(\R^{s'})}\leq   \norm{\sup_{\overline{\mathbf x}\in \R^s} \mathfrak{A}^{{high0}}_\delta f(\overline{\mathbf{x}},\underline{\mathbf x},r)}_{L^p(\R^{s'})}\leq C\norm{f}_{L^p(\R^d)}.
        \end{equation*}
    \end{proposition}
  \begin{proof}
    Suppressing the dependence of $\delta$, we denote by 
    \begin{equation}\label{multi}
        m(\boldsymbol \xi)=  \int_\R e^{-2\pi ir\boldsymbol{\xi}\cdot \gamma(u)}\chi_0(u)\mathfrak{c}(\boldsymbol{\xi})du, 
    \end{equation}
    where $\mathfrak{c}(\boldsymbol{\xi})=\widehat{\psi}(\delta\tilde{\boldsymbol{\xi}})(1-\mathfrak{a}(\boldsymbol{\xi}))(1-\phi(\boldsymbol{\xi}))$ is a smooth bump function supported in 
    \begin{equation*}
        \left\{\boldsymbol{ \xi}\in\R^d: |\boldsymbol{\xi}|>1, |\xi_1|> (100d)^d \left(\sum_{i=2}^d|\xi
        _i|^2\right)^{1/2}\right\}.
    \end{equation*}
    Then by \eqref{high0} and \eqref{multi},
    \begin{equation}\label{multihigh0compact}
        \mathfrak A_\delta^{high0}f(\mathbf x,r)=\chi_1(r)\widecheck{m}*f(\mathbf x).
    \end{equation}
    
    Note that we always assume $f\geq 0$. Apply Littlewood-Paley decomposition to the whole frequency space, for $k\geq 0$
     \begin{equation*}
        \mathcal F( \mathcal P_k f)(\boldsymbol{\xi}):=\varphi_k(\boldsymbol{\xi})f(\boldsymbol{\xi}),
     \end{equation*}
     where $\varphi_k(\boldsymbol{\xi})$ is a smooth cut-off function supported in the annulus $\{\boldsymbol{\xi}: 2^{k-1}<|\boldsymbol{\xi}|\leq 2^{k+1}\}$. Then by triangle inequality, we have the pointwise control 
     \begin{equation*}
        \sup_{\overline{\mathbf x}\in\R^s}  \mathfrak{A}^{high0}_\delta  f(\overline{\mathbf{x}},\underline{\mathbf x},r)\leq \sum_k \sup_{\overline{\mathbf x}\in\R^s}  \mathfrak{A}^{high0}_\delta  \mathcal P_kf(\overline{\mathbf{x}},\underline{\mathbf x},r)
     \end{equation*}
     for any $(\underline{\mathbf x},r)$, which implies $L^p$ bound,
     \begin{equation*}
          \norm{  \sup_{\overline{\mathbf x}\in\R^s}  \mathfrak{A}^{high0}_\delta  f(\overline{\mathbf{x}},\underline{\mathbf x},r)}_{L^p_{\underline{\mathbf x},r}}\leq \sum_k \norm{\sup_{\overline{\mathbf x}\in\R^s}  \mathfrak{A}^{high0}_\delta  \mathcal P_kf(\overline{\mathbf{x}},\underline{\mathbf x},r)}_{L^p_{\underline{\mathbf x},r}}.
     \end{equation*}
     It suffices to prove the the right hand side is summable. In fact, we have the following classical result. 
     \begin{lemma}[integrate by parts]\label{intbypar}With above notations, for any $N>0$, there is $C_N$ such that for any $k\geq 0,$
     \begin{equation}\label{arbdec}
        \norm{\sup_{\overline{\mathbf x}\in\R^s}  \mathfrak{A}^{high0}_\delta  \mathcal P_kf(\overline{\mathbf{x}},\underline{\mathbf x},r)}_{L^p_{\underline{\mathbf x},r}}\leq C_N2^{-k(N-\frac{s}{p})}\norm{f}_{L^p(\R^d)}.
     \end{equation}
     \end{lemma}
    If this is the case, by taking $N=10s$, note $p\geq 1$, we obtain 
     \begin{equation*}
        \sum_k \norm{\sup_{\overline{\mathbf x}\in\R^s}  \mathfrak{A}^{high0}_\delta  \mathcal P_kf(\overline{\mathbf{x}},\underline{\mathbf x},r)}_{L^p_{\underline{\mathbf x},r}}\lesssim_N \sum_k 2^{-k}\norm f_{L^p(\R^d)}\lesssim \norm f_{L^p(\R^d)}
     \end{equation*}
     which concludes the proof.
 \end{proof}
 We now turn to the proof of the Lemma \ref{intbypar}, which is a classical argument in oscillatory integral.
 \begin{proof}[Proof of Lemma \ref{intbypar}] For any fixed $k\geq 0,$ due to Littlewood-Paley decomposition, similar to \eqref{compactsupport1}, it is easy to verify that the Fourier transform of $ \mathfrak{A}^{high0}_\delta  \mathcal P_kf$ (with respect to $\overline{\mathbf x}$) is supported in a ball with radius $ O(2^k)$. By Bernstein's inequality \eqref{bernstein}, 
\begin{equation*}
     \mathfrak{A}^{high0}_\delta  \mathcal P_kf(\overline{\mathbf{x}},\underline{\mathbf x},r)\lesssim  2^{\frac{sk}{p}} \norm{ \mathfrak{A}^{high0}_\delta  \mathcal P_kf(\overline{\mathbf{x}},\underline{\mathbf x},r)}_{L^p_{\overline{\mathbf x}}(\R^s)}.
\end{equation*}
 Raising both sides by $p$-th power and integrating over $(\underline{\mathbf x},r)$, we obtain 
     \begin{equation*}
         \norm{\sup_{\overline{\mathbf x}\in\R^s}  \mathfrak{A}^{high0}_\delta  \mathcal P_kf(\overline{\mathbf{x}},\underline{\mathbf x},r)}_{L^p_{\underline{\mathbf x},r}}\leq2^{\frac{sk}{p}} \norm{ \mathfrak{A}^{high0}_\delta  \mathcal P_kf(\overline{\mathbf{x}},\underline{\mathbf x},r)}_{L^p(\R^{d+1})}.
     \end{equation*}
     In order to prove \eqref{arbdec}, it suffices to prove for any $N$, the right hand side
     \begin{equation*}
        \norm{ \mathfrak{A}^{high0}_\delta  \mathcal P_kf(\overline{\mathbf{x}},\underline{\mathbf x},r)}_{L^p(\R^{d+1})}\leq C_N2^{-kN}\norm{f}_{L^p(\R^d)}.
     \end{equation*}
     Recall \eqref{multihigh0compact}, combining with the cut-off function $\varphi_k$, we find the new multiplier for  $ \mathfrak{A}^{high0}_\delta  \mathcal P_kf$ is 
    \begin{equation*}
          \chi_1(r)m_k(\boldsymbol \xi)=  \chi_1(r)m(\boldsymbol \xi)\varphi_k(\boldsymbol{\xi})=\chi_1(r)\mathfrak{c}(\boldsymbol{\xi})\varphi_k(\boldsymbol{\xi})\int_\R e^{-2\pi ir\boldsymbol{\xi}\cdot \gamma(u)}\chi_0(u)du.
    \end{equation*}  
     Similar to \eqref{dropr}, after integrating over $r$ first, it suffices to prove 
     \begin{equation*}
         \norm{ \mathfrak{A}^{high0}_\delta  \mathcal P_kf(\overline{\mathbf{x}},\underline{\mathbf x},r)}_{L^p(\R^{d+1})}\lesssim\norm{\widecheck m_k *f}_{L^p(\R^d)}\leq C_N2^{-kN}\norm{f}_{L^p(\R^d)}.
     \end{equation*}
     By Young's convolution inequality, it suffices to prove
\begin{equation}\label{newasd}
\norm{\widecheck{m_k}}_{L^1(\R^d)}\leq C_N 2^{-kN}.
\end{equation}
In fact, this is just integration by parts, which  is repeatedly used in Chapter VII, VIII of Stein's book \cite{bible} when discussing the decay of oscillatory integrals.

We first fix $\mathbf x: |\mathbf x|\geq 2^{-k(N+1)}$, define differential operator $D_{\boldsymbol{\xi}}^\mathbf x$ as 
\begin{equation*}
    D_{\boldsymbol{\xi}}^\mathbf x(f)(\boldsymbol{\xi}):=\frac{1}{2\pi i |\mathbf x|^2}\nabla_{\boldsymbol \xi}\cdot \mathbf x.
\end{equation*}
It can be checked that 
\begin{equation*}
    D_{\boldsymbol{\xi}}^\mathbf{x}(e^{2\pi i \boldsymbol{\xi}\cdot \mathbf x})=e^{2\pi i \boldsymbol{\xi}\cdot \mathbf x}.
\end{equation*}
As a result, for any $M\in \mathbb N_+$,
\begin{equation*}
    \widecheck{m_k}(\mathbf x)=\int m_k(\boldsymbol{\xi})e^{2\pi i \boldsymbol{\xi}\cdot \mathbf x}d\boldsymbol{\xi}=\int m_k(\boldsymbol{\xi})(D_{\boldsymbol{\xi}}^\mathbf x)^M(e^{2\pi i \boldsymbol{\xi}\cdot \mathbf x})d\boldsymbol{\xi}. 
\end{equation*}
It can be verified by integration by parts,  that if either $f$ or $g$ is a compactly supported function, then 
\begin{equation*}
    \langle f, D_{\boldsymbol{\xi}}^\mathbf x g \rangle= -\langle D_{\boldsymbol{\xi}}^\mathbf xf,  g \rangle,
\end{equation*}
where $  \langle \cdot, \cdot \rangle$ denotes the integral of the product of two functions. Due to the compactness of $m_k$, we can transfer $(D_{\boldsymbol{\xi}}^{\mathbf x})^M$ to $m_k$ and limit the domain of integration in a ball with radius $O(2^k)$. So 
\begin{equation}\label{mk}
\begin{aligned}
      \widecheck{m_k}(\mathbf x)&=(-1)^M\int (D_{\boldsymbol{\xi}}^\mathbf x)^M(m_k)(\boldsymbol{\xi})e^{2\pi i \boldsymbol{\xi}\cdot \mathbf x}d\boldsymbol{\xi}\\
      &=(-1)^M\int_{B(\mathbf 0, O(2^k))} (D_{\boldsymbol{\xi}}^\mathbf x)^M(m_k)(\boldsymbol{\xi})e^{2\pi i \boldsymbol{\xi}\cdot \mathbf x}d\boldsymbol{\xi}.
\end{aligned}
\end{equation}
We claim 
\begin{lemma}[fast decay for $(D_{\boldsymbol{\xi}}^\mathbf x)^M(m_k)(\boldsymbol{\xi})$]\label{fast decay for m_k}For any $J\in \mathbb N_+$, there is a constant $C_{M,J}$ depending on $M$, $J$ and dimension, such that 
\begin{equation}\label{estfordm}
    |(D_{\boldsymbol{\xi}}^\mathbf x)^M(m_k)(\boldsymbol{\xi})|\leq C_{M,J} \frac{1}{|\mathbf x|^{M}} 2^{-k(J-M)}.
\end{equation}
\end{lemma}
We will put the proof of this lemma in Appendix \ref{app fast decay for m_k}.
Take the lemma for granted and plug the estimate \eqref{estfordm} back to \eqref{mk}, we obtain 
\begin{equation*}
    |\widecheck{m_k}|\leq C_{M,J}2^{kd}\frac{1}{|\mathbf x|^{M}} 2^{-k(J-M)}
\end{equation*}
Note that $|\mathbf x|\geq 2^{-k(N+1)}$, as long as $M,J$ are chosen to be sufficiently large and appropriately (depending on $d$ and $N$), 
\begin{equation}\label{part1}
\begin{aligned}
    \int_{\{\mathbf x:  |\mathbf x|\geq 2^{-k(N+1)}\}}  |   \widecheck{m_k}(\mathbf x)| d\mathbf x&\lesssim_{N,M,J} 2^{kd}2^{-k(J-M)}\int_{\{\mathbf x:  |\mathbf x|\geq 2^{-k(N+1)}\}}\frac{1}{|\mathbf x|^M}d\mathbf x \\&\leq   C_N 2^{-kN}.
\end{aligned}
\end{equation}
For $\mathbf x:  |\mathbf x|< 2^{-k(N+1)}$, by a basic property of Fourier transform and supp$(\varphi_k)\subseteq B(\mathbf 0,O(2^k))$, we obtain 
\begin{equation*}
    \norm{\widecheck{m_k}}_{L^\infty}\leq \norm{m_k}_{L^1(\R^d)}\lesssim \norm{\varphi_k}_{L^1(\R^d)}\lesssim 2^{kd}.
\end{equation*}
So 
\begin{equation}\label{part2}
      \int_{\{\mathbf x:  |\mathbf x|< 2^{-k(N+1)}\}}  |   \widecheck{m_k}(\mathbf x)| d\mathbf x\leq C_N 2^{-kdN}\leq C_N 2^{-kN}.
\end{equation}
Combining \eqref{part1} and \eqref{part2}, we get the desired \eqref{newasd}.  
 \end{proof}

\subsection{Estimating $\M_\delta^{s,high1} f=\sup \mathfrak{A}^{{high1}}_\delta f$}\label{seclocalsmoothing}

\subsubsection{A local smoothing estimate}In this part, we introduce the setting and the result of Ko, Lee and Oh \cite{Ko_Lee_Oh_2023}.

Let $\Gamma: I\to \R^d$ be a smooth curve and $B\geq 1$ be a large number. Assume $\Gamma$ satisfies the following two conditions 
\begin{align}
    \label{aa}\max_{0\leq j\leq 3d+1}|\Gamma^{(j)}&(u)|\leq B,\quad \forall u\in I,\\ \label{bb}
    \mathrm{Vol}(\Gamma^{(1)}(u),...,&\Gamma^{(d)}(u))\geq 1/B,\quad \forall u\in I,
\end{align}
where $\mathrm{Vol}(v_1,...,v_d)$ denotes the $d$-dimensional volume of the parallelepiped generated by $v_1,...v_d\in \R^d$.

\begin{definition}\label{def7}
    
For $k \geq 0$, let $\mathbb{A}_k = \{\xi \in \mathbb{R}^d : c2^{k-1} \leq |\xi| \leq C2^{k+1}\}$. We say $a \in C^{2d+2}(\mathbb{R}^{d+2})$ is a \textit{symbol of type $(k,d, B)$ relative to $\Gamma$} if $\operatorname{supp} a \subset I \times [2^{-1},4] \times \mathbb{A}_k$, and 

\begin{equation}\label{symboldecaycondition}
    |\partial_s^j \partial_t^l \partial_\xi^\alpha a(s, t, \xi)| \leq B |\xi|^{-|\alpha|}
\end{equation}
for $(j, l, \alpha) \in \mathcal{I}_d := \{(j, l, \alpha) : 0 \leq j \leq 1, 0 \leq l \leq 2d, |\alpha| \leq 2d  + 2\}$.
\end{definition} 
We define an integral operator by

\[
\mathcal{A}_t[\Gamma, a] f(x) = (2\pi)^{-d} \iint e^{i(x - t \Gamma(s)) \cdot \xi} a(s, t, \xi) ds \, \hat{f}(\xi) d\xi. 
\]

Then we have

\begin{theoremB}[{\cite[Theorem 2.2]{Ko_Lee_Oh_2023}}, local smoothing estimate]\label{localsmoothing}
 Let $\Gamma \in C^{3d+1}(I)$ satisfy \eqref{aa} and \eqref{bb} for some $B \geq 1$. Suppose $a$ is a symbol of type $(k, d, B)$ relative to $\Gamma$. Then, if $p \geq 4d - 2$, for $\epsilon > 0$ there is a constant $C_\epsilon = C_\epsilon(B)$ such that

\[
\|\mathcal{A}_t[\Gamma, a] f\|_{L^p_{\mathbf x,r}(\mathbb{R}^{d+1})} \leq C_\epsilon \, 2^{\left(-\frac{2}{p} + \epsilon\right) k} \|f\|_{L^p(\mathbb{R}^d)}. 
\]
\end{theoremB}

\subsubsection{Maximal estimate via local smoothing} Now we adapt our problem to their setting. Recall the definition of $m_{\delta,r}^{high1}$, the multiplier of $\mathfrak{A}^{{high1}}_\delta f$ from \eqref{high1}. 

Applying Littlewood-Paley decomposition to $\tilde{\boldsymbol{\xi}}=(\xi_2,...,\xi_d)$ in frequency space and assuming that the smooth cut-off functions are $\varphi_k(\tilde{\boldsymbol{\xi}})$ which is supported in $\sum_{i=2}^d|\xi_i|^2\in  (2^{2(k-1)},2^{2(k+1)}], k\geq 1$, we also denote the corresponding Littlewood-Paley projection by $\mathcal{P}_k$. I.e. we have 
\begin{equation}\label{multiplier version}
\begin{aligned}
           \mathfrak{A}^{high1}_\delta  &\mathcal P_kf(\mathbf x,r)=\\
          \chi_1(r)&\int_{\R^d} \widehat{f}(\boldsymbol{\xi})      
           \left(\int_\R e^{-2\pi ir\boldsymbol{\xi}\cdot \gamma(u)}\widehat{\psi}(\delta\tilde{\boldsymbol{\xi}})\chi_0(u)du \right)e^{2\pi i\boldsymbol{\xi}\cdot \mathbf x}\varphi_k(\tilde{\boldsymbol{\xi}})\phi(\boldsymbol{\xi}) d\boldsymbol{\xi}. 
\end{aligned}
\end{equation}
Note that $\phi$ is the smooth cut-off function such that $|\xi_1|\leq 2(100d)^d |\tilde{\boldsymbol{\xi}}|$. 

By the dyadic decomposition of $(\xi_2,...,\xi_d)$, 
it suffices to consider those $k$'s such that $2^{k+1}\leq \delta^{-1}$, I.e. $k\lesssim |\log \delta|$.
Since for the terms with larger $k$'s, they are identically zero due to the compactness of $\widehat \psi$. 
So 
$$ \mathfrak{A}^{high1}_\delta  f=\sum_{k=1}^{\infty}\mathfrak{A}^{high1}_\delta \mathcal{P}_kf= \sum_{k=1}^{O(|\log \delta|)}\mathfrak{A}^{high1}_\delta \mathcal{P}_k f.$$
Applying triangle inequality and taking supremum on both sides, we have
\begin{equation*}
    \norm{\sup_{\overline{\mathbf x}\in \R^s}\mathfrak A^{high1}_\delta f}_{L^p(\R^{s'})}\leq \sum_{k=1}^{O(|\log\delta|)} \norm{\sup_{\overline{\mathbf x}\in \R^s}       \mathfrak{A}^{high1}_\delta  \mathcal P_kf}_{L^p(\R^{s'})}
\end{equation*}
Note that the compactness of $r$ is included in the definition of the averaging operator.
It suffices to consider each term of the summand on the right hand side. I.e. it suffices to prove for each $\eps>0$, there is a constant $C_\eps>0$ such that for all $k$, 
\begin{equation}\label{eq111}
    \norm{\sup_{\overline{\mathbf x}\in \R^s}\mathfrak A_\delta^{high1}  \mathcal P_kf}_{L^p(\R^{s'})}\leq C_\eps\delta^{-\eps} \norm{f}_{L^p(\R^d)}.
\end{equation}
If this is true, then 
\begin{equation*}
\begin{aligned}
     \sum_{k=1}^{O(|\log\delta|)} \norm{\sup_{\overline{\mathbf x}\in \R^s}       \mathfrak{A}^{high1}_\delta  \mathcal P_kf}_{L^p(\R^{s'})}&\leq C_\eps O|\log\delta|\delta^{-\eps}\norm{f}_{L^p(\R^d)}\\
     &\leq C_\eps\delta^{-\eps}\norm{f}_{L^p(\R^d)}
\end{aligned}
\end{equation*}
since $\eps$ is arbitrary.

By the previous computations, we have $2^k\lesssim \delta^{-1}$. This implies $2^{k\eps}\lesssim_\eps \delta^{-\eps}$ so to prove \eqref{eq111}, it suffices to prove that 
\begin{equation}\label{reduction equation}
    \norm{\sup_{\overline{\mathbf x}\in \R^s}\mathfrak A_\delta^{high1}  \mathcal P_kf}_{L^p(\R^{s'})}\leq C_\eps 2^{k\eps}\norm{f}_{L^p(\R^d)}.
\end{equation}
Now we are ready to use this set up to verify the conditions of \cite{Ko_Lee_Oh_2023}. Since $2^{k+1}\leq\delta^{-1},$ supp$(\varphi_k)\subseteq$ supp$(\widehat{\psi}(\delta\cdot))$ We rewrite equation \eqref{multiplier version} as 

\begin{equation*}
\begin{aligned}
      \mathfrak{A}^{high1}_\delta  &\mathcal{P}_k f(\mathbf x,r)\\
     &=\chi_1(r)\int_{\R^d} \widehat{f}(\boldsymbol{\xi})\left(\int_\R e^{-2\pi ir\boldsymbol{\xi}\cdot \gamma(u)}\chi_0(u)du \right)
     \times e^{2\pi i\boldsymbol{\xi}\cdot \mathbf x}\varphi_k(\tilde{\boldsymbol{\xi}})\phi(\boldsymbol{\xi})d\boldsymbol{\xi}\\
        &=\int\int_\R e^{2\pi i\boldsymbol{\xi}\cdot (\mathbf x-r\gamma(u))}a(u,r,\boldsymbol{\xi})du\widehat{f}(\boldsymbol{\xi})d\boldsymbol{\xi},
\end{aligned}
\end{equation*}
where 
\[
    a(u,r,\boldsymbol{\xi})=\chi_0(u)\chi_1(r)\phi(\boldsymbol{\xi})\varphi_k(\tilde{\boldsymbol{\xi}}).
\]
It is easy to verify the requirement on the support of $a$ is satisfied.
To use their conclusion, we need $a$ to satisfy  
\eqref{symboldecaycondition}. This is achieved by selecting those bump functions carefully. We give an intuitive explanation for the computations. For the partial derivatives of $u$ and $r$, since $\chi_0$ and $\chi_1$ are just fixed bump functions, we can always choose $B$ sufficiently large to tackle this part. For the decay of $\boldsymbol{\xi}$, combining the effect of $\phi$ and $\varphi_k$, this function is roughly a smooth bump function supported in a annulus with radius $2^k$ and satisfies \eqref{symboldecaycondition} naturally. We will give a computation for this part in Appendix \ref{Verification of conditions}.

To apply their local smoothing, Theorem B, we also need to verify the following conditions. By choosing a sufficiently large $B$, it is easy to see that the truncated moment curve $\gamma \in C^{3d+1}$, and satisfies \eqref{aa} and \eqref{bb}:
    \begin{align*}
        \max_{0\leq i\leq 3d+1}|\gamma^{(i)}(u)|\leq (3d+1)(2d)^{100d}&\leq B\\
        \text{Vol}(\gamma^{(1)}(u),...,\gamma^{{(d)}}(u))=\prod_{i=1}^di!&\geq 1/B.
    \end{align*}
     Then for $p\geq4d-2$ and all $\eps>0$, there is a $C_\eps=C_\eps(B)$ such that 
    \begin{equation}\label{local smoothing}
        \norm{ \mathfrak{A}^{high1}_\delta \mathcal{P}_kf(\mathbf x,r)}_{L^p_{\mathbf x,r}(\R^{d+1})}\leq C_\eps 2^{(-\frac{2}{p}+\eps)k}\norm{f}_{L^p(\R^d)}.
    \end{equation}

With this theorem we can continue our computations to get \eqref{reduction equation}. 

Fix $(x_{s+1},...,x_d,r)=(\underline{\mathbf{x}},r)$, by computing the Fourier transform of the function $$F(\overline{\mathbf x})=F_{\underline{\mathbf x},r}(\overline{\mathbf x})= \mathfrak{A}^{high1}_\delta \mathcal{P}_kf(\overline{\mathbf x},\underline{\mathbf x},r)$$ with respect to $\overline{\mathbf x}=(x_1,...,x_s)$, note that $|\xi_1|\leq 2(100d)^d |\tilde{\boldsymbol{\xi}}|$ due to the existence of $\phi$,  we can find that the Fourier transform $\widehat{F}(\overline{\boldsymbol{\xi}})$ is compactly supported in a ball with radius $O(2^{k})$. Applying Bernstein's inequality \eqref{bernstein} to $F$, we have
\begin{equation*}
    \sup_{\overline{\mathbf x}\in\R^s} F\lesssim2^{\frac{sk}{p}}\norm{F}_{L^p(\R^s)}.
\end{equation*}
Raising both sides by $p$-th power and integrating over $(\underline{\mathbf x},r)$, we obtain
\begin{equation*}
     \norm{\sup_{\overline{\mathbf x}}\mathfrak A^{high1}_\delta \mathcal P_kf}_{L^p(\R^{s'})}\lesssim 2^{\frac{sk}{p}}\norm{ \mathfrak{A}^{high1}_\delta  \mathcal{P}_kf(\mathbf x,r)}_{L^p(\R^{d+1})}.
\end{equation*}
Then we can apply the local smoothing estimate \eqref{local smoothing}
\begin{align*}
 \norm{\sup_{\overline{\mathbf x}}\mathfrak A^{high1}_\delta \mathcal P_kf}_{L^p(\R^{s'})}&\lesssim 2^{\frac{sk}{p}} \times 2^{-\frac{2k}{p}+\eps k}\norm{f}_{L^p(\R^d)}\\
    &\lesssim 2^{\frac{(s-2)k}{p}+\eps k}\norm{f}_{L^p(\R^d)}\\
    &\lesssim \delta^{\frac{(s-2)}{p}-\eps}\norm{f}_{L^p(\R^d)},
\end{align*}
for $p\geq 4d-2$, which is \eqref{reduction equation}. This concludes the proof.

Note that when $s=1$, the previous two cases dominate. Hence we cannot anticipate any bound $\to 0$ as $\delta\to 0.$

\section*{Appendix}
\addcontentsline{toc}{section}{Appendix}
 \subsection{The cut-off functions and smooth version domination}
 In this section, we prove the existence of the cut-off functions and the smooth version indeed can bound the rough original averaging operator. I.e. \eqref{smoothdomination} holds.
 \subsubsection{Existence of the cut-off functions}\label{existenceofcutoff} The existence of $\chi_1$ and $\chi_0$ is trivial. We omit the proof. So it suffices to prove the existence of $\psi.$ 
 \begin{proposition}
There is a function $\psi(\mathbf x):\R^{d}\to \R$ is a nonnegative smooth function that satisfies that 
\begin{enumerate}
    \item $\hat{\psi}$ is compactly supported and nonnegative.
    \item $\int_{\R^{d}} \psi>0$.
    \item\label{item3} $\psi(\mathbf x)\geq 1$ for some fixed $\mathbf x\in B(\mathbf 0,C_d)$ for some $C_d>0$ to be determined later.
\end{enumerate}     
 \end{proposition}
 \begin{proof}
     Take a non-zero nonnegative smooth function $\hat{\phi}$ with compact support contained in $B(\mathbf 0,1/2).$ Then  $\mathcal F^{-1}(\hat{\phi}*\bar{\hat{\phi}})=\phi^2$ is nonnegative. Furthermore, supp$(\hat{\phi}*\bar{\hat{\phi}})\subseteq B(\mathbf 0,1)$ and by definitions of convolution and conjugate, $\hat{\phi}*\bar{\hat{\phi}}=\hat{\phi}*\hat{\phi}$ is also non-negative, non-zero.
     By multiplying a factor of $\frac{100}{\int \hat{\phi}*\bar{\hat{\phi}}}$ if necessary, we can assume $\int \hat{\phi}*\bar{\hat{\phi}}=c$ for some $c>10.$ This implies that 
     \begin{equation*}
         \phi^2(0)=\int \hat{\phi}*\bar{\hat{\phi}}=c>10.
     \end{equation*}
     By continuity, there is a $\eps_o>0$ depending on $\hat \phi$ such that for all $|\mathbf x|\leq \eps_o$,
     \begin{equation*}
         \phi^2(\mathbf x)\geq 5.
     \end{equation*}
     Rescaling the function, we have 
     \begin{equation}\label{rescaling}
         \phi^2(\frac{\eps_o \mathbf x}{C_d})\geq 5
     \end{equation} for $\mathbf x\in B(\mathbf 0,C_d)$. Define 
    \begin{equation*}
        \psi(\mathbf x):= \phi^2(\frac{\eps_o \mathbf x}{C_d}).
    \end{equation*}
    Apparently its integral is positive so item $(2)$ is satisfied and item $(3)$ is guaranteed by \eqref{rescaling}. For item $(1)$, it suffices to notice that the dilation factor $\frac{\eps_o \mathbf x}{C_d}$ will not affect the compactness in frequency space and the non-negativity of the function. 
 \end{proof}
 \subsubsection{Smooth version domination}\label{smoothversiondomination} In this part we prove \eqref{smoothdomination}. For convenience, recall our goal is to prove for $r\in [1/2,2]$
 \begin{equation}\label{292}
    \frac{1}{\mathcal L^d( H_\delta(\mathbf x,r))}\int_{H_\delta(\mathbf x,r)}|f|\lesssim  \chi_1(r)\int_{\R^d}f(\mathbf x-r\gamma(u)-(0,\tilde{\boldsymbol{t}}))\psi_{\delta}(\tilde{\boldsymbol{t}})d\tilde{\boldsymbol{t}}\chi_0(u)du.
\end{equation}
\begin{proof}[Proof of \eqref{292}]
    Assume $f\geq 0$, up to some constant, by change of variable $\mathbf y\to \mathbf y+\mathbf x$, we have the left hand side 
 \begin{equation*}
      \frac{1}{\mathcal L^d( H_\delta(\mathbf x,r))}\int_{H_\delta(\mathbf x,r)}f(\mathbf y)d\mathbf y \lesssim \frac{1}{\delta^{d-1}}\int_{H_\delta(\mathbf 0,r)}f(\mathbf x+\mathbf y)d\mathbf y. 
 \end{equation*}
 We claim 
 \begin{lemma}
     $H_\delta(\mathbf 0,r)\subseteq \{r\gamma(u)+\{0\}\times B^{d-1}(\mathbf 0,C_d\delta): u\in [-1.2,1.2]\}=:H_\delta^{ext}(r)$, where $B^{d-1}(\mathbf 0,C_d)$ is a $d-1$ dimensional ball and $C_d$ is a constant that depends on $d$.
 \end{lemma}
 \begin{proof}
     For each $\mathbf y\in H_\delta(\mathbf 0,r),\ \exists t\in[-1,1]$ such that 
        \begin{equation}\label{deltanbhd}
         \mathbf y-r\gamma(t)=\mathbf s,
     \end{equation}
     where $|\mathbf s|\leq\delta$. In particular,
     \begin{equation}\label{y1}
         y_1=rt+s_1=r(t+\frac{s_1}{r}).
     \end{equation}
     let $u=t+\frac{s_1}{r}$ then $|u|=|t+\frac{s_1}{r}|\leq 1+2\delta\leq 1.2$ as long as $\delta$ is small enough.
     By triangle inequality, we obtain
     \begin{equation*}
         |\mathbf y-r\gamma(u)|\leq|\mathbf y-r\gamma(t)|+r|\gamma(t)-\gamma(u)|.
     \end{equation*}
     The first term is $\leq\delta$ due to \eqref{deltanbhd}. For the second term, by the smoothness of $\gamma$ and mean value theorem, it is smaller than $c_d|t-u|\leq 2c_d\delta$ for some constant $c_d$ depending only on dimension. Let $C_d=2c_d+1$. Then $  |\mathbf y-r\gamma(u)|\leq C_d\delta$. By the definition of $\mathbf y$ and \eqref{y1}, we have 
     \begin{equation*}
         y_1=ru,
     \end{equation*}
     which means that if we denote $\mathbf y=r\gamma(u)+\mathbf s'$, then $\mathbf s'=\mathbf y-r\gamma(u)=(0,\tilde{\mathbf s'})\in \{0\}\times B^{d-1}(\mathbf 0,C_d\delta)$ where $\tilde{\mathbf s'}$ is the last $d-1$ coordinates of $\mathbf{s'}$. This concludes the proof.
 \end{proof}
Back to our computation, enlarging the domain of the integration, we have 
 \begin{equation*}
     \frac{1}{\delta^{d-1}}\int_{H_\delta(\mathbf 0,r)}f(\mathbf x+\mathbf y)d\mathbf y\lesssim \frac{1}{\delta^{d-1}}\int_{H_\delta^{ext}(r)}f(\mathbf x+\mathbf y)d\mathbf y .
 \end{equation*}
 For the right hand side, since 
 \begin{equation}\label{twist}
     \mathbf y=r\gamma(u)+{\mathbf s'} =r\gamma(u)+(0,\tilde{\mathbf s'}),
 \end{equation}
 It is easy to check that 
 \begin{equation*}
     d\mathbf y= r dud\tilde{\mathbf s'}\leq 2dud\tilde{\mathbf s'},
 \end{equation*}
 In fact, the Jacobian matrix of the map \eqref{twist} is 
 \[
\begin{pmatrix}
r & 0 & 0 & \cdots & 0 \\
2ru & 1 & 0 & \cdots & 0 \\
3ru^2 & 0 & 1 & \cdots & 0 \\
\vdots & \vdots & \vdots & \ddots & \vdots \\
d r u^{d-1} & 0 & 0 & \cdots & 1
\end{pmatrix}
\]

 By change of variable and Fubini-Tonelli, we expand the integral and use characteristic functions to limit the domain of integration,
\begin{equation}\label{rhssss}
\begin{aligned}
     \frac{1}{\delta^{d-1}}&\int_{H_\delta^{ext}}f(\mathbf x+\mathbf y)d\mathbf y\\ &\lesssim \frac{1}{\delta^{d-1}}\int_{\R^d}f(\mathbf x+r\gamma(u)+(0,\tilde{\mathbf s'}))\chi_{[-1.2,1.2]}(u)\chi_{B^{d-1}(\mathbf 0,C_d\delta)}(\tilde{\mathbf s'})dud\tilde{\mathbf s'}\\
     &=\frac{1}{\delta^{d-1}}\int_{\R^d}f(\mathbf x-r\gamma(u)-(0,\tilde{\mathbf s'}))\chi_{[-1.2,1.2]}(u)\chi_{B^{d-1}(\mathbf 0,C_d\delta)}(\tilde{\mathbf s'})dud\tilde{\mathbf s'}.
\end{aligned}
\end{equation}
Recall $\chi_1(r):\R\to \R$ is a non-negative function that is
compactly supported in $[1/4,3]$ and $\chi_1(r)=1 $ for $r\in[1/2,1]$ and  $\chi_0(u):\R\to \R$ is a non-negative function that is
compactly supported in $[-2,2]$ and $\chi_0(u)=1 $ for $u\in [-1.5,1.5]$.
So trivially we have,
\begin{align*}
    \chi _{[1/2,2]}(r)&\lesssim \chi_1(r)\\
    \chi_{[-1.2,1.2]}(u)&\lesssim \chi_0(u).
\end{align*}
Comparing the right hand side of \eqref{rhssss} with the right hand side of \eqref{292}, it suffices to prove 
\begin{equation*}
    \chi_{B^{d-1}(\mathbf 0,C_d\delta)}(\tilde{\mathbf s'})\lesssim \psi(\frac{\tilde{\mathbf s'}}{\delta}),
\end{equation*}
which is immediate from item $\eqref{item3}$ of the property of $\psi.$ So we conclude the proof of \eqref{smoothdomination}.
\end{proof}
\subsection{Proof of Lemma \ref{fast decay for m_k}}\label{app fast decay for m_k} Let us recall that our goal is to find some $C_{M,J}$ such that 
\begin{equation*}
       |(D_{\boldsymbol{\xi}}^\mathbf x)^M(m_k)(\boldsymbol{\xi})|\leq C_{M,J} \frac{1}{|\mathbf x|^{M}} 2^{-kJ},
\end{equation*}
where 
  \begin{equation*}
          m_k(\boldsymbol \xi)=  \int_\R e^{-2\pi ir\boldsymbol{\xi}\cdot \gamma(u)}\chi_0(u)\mathfrak{c}(\boldsymbol{\xi})\varphi_k(\boldsymbol{\xi})du.
    \end{equation*}
\begin{proof}[Proof of Lemma \ref{fast decay for m_k}]
   In what following, the constant $C_M$ and other constants may change from line to line. By Leibniz's rule,
   \begin{equation}\label{dmk}
       \begin{aligned}
(&D_{\boldsymbol{\xi}}^\mathbf x)^M(m_k)(\boldsymbol{\xi})\\&=\frac{1}{(2\pi i |\mathbf x|^2)^M}\sum_{j=0}^MC_{j,M}\int (\nabla_{\boldsymbol \xi}\cdot \mathbf x)^j e^{-2\pi i \boldsymbol{\xi}\cdot \gamma(u)}(\nabla_{\boldsymbol \xi}\cdot \mathbf x)^{M-j}(\mathfrak{c}(\boldsymbol{\xi})\varphi_k(\boldsymbol{\xi}))\chi(u)du\\
    &=\frac{C}{|\mathbf x|^{2M}}\sum_{j=0}^MC_{j,M}\int (-2\pi i \mathbf x\cdot \gamma(u))^j e^{-2\pi i \boldsymbol{\xi}\cdot \gamma(u)}(\nabla_{\boldsymbol \xi}\cdot \mathbf x)^{M-j}(\mathfrak{c}(\boldsymbol{\xi})\varphi_k(\boldsymbol{\xi}))\chi(u)du.
\end{aligned}
   \end{equation}

Since the number of the terms is $M+1$ and the coefficient $C_{j,M}$ (in fact they are binomial coefficients) depends only on $j$ and $M$. So it suffices to consider each fixed $j$-th term and take the final $C_{M,J}$ as the maximum of coefficients we get from each term.

Fix $j$, since the integral is over variable $u$, we can extract the $\boldsymbol{\xi}$ part.
\begin{align*}
    &\int (-2\pi i \mathbf x\cdot \gamma(u))^j e^{-2\pi i \boldsymbol{\xi}\cdot \gamma(u)}(\nabla_{\boldsymbol \xi}\cdot \mathbf x)^{M-j}(\mathfrak{c}(\boldsymbol{\xi})\varphi_k(\boldsymbol{\xi}))\chi(u)du\\
    =&\left(\int (-2\pi i \mathbf x\cdot \gamma(u))^j e^{-2\pi i \boldsymbol{\xi}\cdot \gamma(u)}\chi(u)du\right)(\nabla_{\boldsymbol \xi}\cdot \mathbf x)^{M-j}(\mathfrak{c}(\boldsymbol{\xi})\varphi_k(\boldsymbol{\xi})).
\end{align*}
Applying change of variable $\boldsymbol{\xi}\to 2^k\boldsymbol{\xi}$, we obtain 
\begin{align*}
    &\left(\int (-2\pi i \mathbf x\cdot \gamma(u))^j e^{-2\pi i \boldsymbol{\xi}\cdot \gamma(u)}\chi(u)du\right)(\nabla_{\boldsymbol \xi}\cdot \mathbf x)^{M-j}(\mathfrak{c}(\boldsymbol{\xi})\varphi_k(\boldsymbol{\xi}))\\
    =&\left(\int (-2\pi i \mathbf x\cdot \gamma(u))^j e^{-2\pi i 2^k \boldsymbol{\xi}\cdot \gamma(u)}\chi(u)du\right)(\nabla_{\boldsymbol \xi}\cdot \mathbf x)^{M-j}(\mathfrak{c}(2^k\boldsymbol{\xi})\varphi_0(\boldsymbol{\xi})).
\end{align*}
Note that the domain of the integration is $\{\boldsymbol{\xi}:|\boldsymbol{\xi}|\in (1/2,2]\}$ and differentiation does not affect the support of the functions, so 
\begin{equation}\label{intbyparts2}
\begin{aligned}
    &\left(\int (-2\pi i \mathbf x\cdot \gamma(u))^j e^{-2\pi i 2^k \boldsymbol{\xi}\cdot \gamma(u)}\chi(u)du\right)(\nabla_{\boldsymbol \xi}\cdot \mathbf x)^{M-j}(\mathfrak{c}(2^k\boldsymbol{\xi})\varphi_0(\boldsymbol{\xi}))\\
    =&\left(\int (-2\pi i \mathbf x\cdot \gamma(u))^j e^{-2\pi i 2^k \boldsymbol{\xi}\cdot \gamma(u)}\chi(u)du(\mathfrak{c}(2^k\boldsymbol{\xi})\varphi_0(\boldsymbol{\xi}))\right)\\ &\times (\nabla_{\boldsymbol \xi}\cdot \mathbf x)^{M-j}(\mathfrak{c}(2^k\boldsymbol{\xi})\varphi_0(\boldsymbol{\xi})).
\end{aligned}
\end{equation}

We first tackle the 
$(\nabla_{\boldsymbol \xi}\cdot \mathbf x)^{M-j}(\mathfrak{c}(2^k\boldsymbol{\xi})\varphi_0(\boldsymbol{\xi}))$ part. For this part we have the following conclusion, 
\begin{lemma}\label{secondpartlemma} There is a constant $C_{j,M},$ such that
    \begin{equation}\label{second part of int by parts}
        |(\nabla_{\boldsymbol \xi}\cdot \mathbf x)^{M-j}(\mathfrak{c}(2^k\boldsymbol{\xi})\varphi_0(\boldsymbol{\xi}))|\leq C_{j,M} 2^{kM}|\mathbf x|^{M-j}.
    \end{equation}
\end{lemma}
\begin{proof}
By Leibniz's rule,  
\begin{equation}\label{leibniz2}
    (\nabla_{\boldsymbol \xi}\cdot \mathbf x)^{M-j}(\mathfrak{c}(2^k\boldsymbol{\xi})\varphi_0(\boldsymbol{\xi}))=\sum_{i}C_{i,j,M} (\nabla_{\boldsymbol \xi}\cdot \mathbf x)^{M-j-i}\mathfrak{c}(2^k\boldsymbol{\xi})(\nabla_{\boldsymbol \xi}\cdot \mathbf x)^i\varphi_0(\boldsymbol{\xi}).
\end{equation}
Similar to what we stated at the beginning of the proof, $i$ is essentially dependent on $j$ and $M$. Therefore it suffices to consider one fixed term $(\nabla_{\boldsymbol \xi}\cdot \mathbf x)^{M-j-i}\mathfrak{c}(2^k\boldsymbol{\xi})(\nabla_{\boldsymbol \xi}\cdot \mathbf x)^i\varphi_0(\boldsymbol{\xi})$. After computations, we obtain 
\begin{equation*}
    (\nabla_{\boldsymbol \xi}\cdot \mathbf x)^{M-j-i}\mathfrak{c}(2^k\boldsymbol{\xi})=2^{k(M-i-j)}\sum_{\substack{|\alpha|=M-i-j\\|\beta|=M-i-j}}(\nabla_{\boldsymbol{\xi}}^{\alpha}\mathfrak{c})(2^k\boldsymbol{\xi}){\mathbf x}^{\beta},
\end{equation*}
and 
\begin{equation*}
    (\nabla_{\boldsymbol \xi}\cdot \mathbf x)^i\varphi_0(\boldsymbol{\xi})=\sum_{\substack{|\alpha|=i\\|\beta|=i}}(\nabla_{\boldsymbol{\xi}}^{\alpha}\varphi_0)(\boldsymbol{\xi}){\mathbf x}^{\beta}.
\end{equation*}
Where $\alpha=(\alpha_1,...,\alpha_d),\ \beta=(\beta_1,...,\beta_d)$ are multi-indices and as conventions, $|\alpha|=\alpha_1+...+\alpha_d$ $\nabla_{\boldsymbol{\xi}}^\alpha=(\p_1^{\alpha_1},...,\p_d^{\alpha_d})$ and $\mathbf x^\beta=(x_1^{\beta_1},...,x_d^{\beta_d})$. 

Since the number of the terms in the summation depends on $i,\ j$ and $M$ and all partial derivatives of $\mathfrak{c}$ and $\varphi_0$ are bounded (bounds may depend on the order $\alpha$ hence can be bounded uniformly by a constant depending on $M$). Therefore we can find a constant $C_{i,j,M}$ such that 
\begin{align*}
    | (\nabla_{\boldsymbol \xi}\cdot \mathbf x)^{M-j-i}\mathfrak{c}(2^k\boldsymbol{\xi})|\times |(\nabla_{\boldsymbol \xi}\cdot \mathbf x)^i\varphi_0(\boldsymbol{\xi})|&\leq C_{i,j,M} 2^{k(M-i-j)}|\mathbf x|^{|\alpha|+|\beta|}\\
    &=C_{i,j,M} 2^{k(M-i-j)}|\mathbf x|^{M-j}.
\end{align*}
Plug this back to \eqref{leibniz2}, we obtain 
\begin{equation*}
    |\mathrm{LHS~of~\eqref{leibniz2}}|\leq \sum_i C_{i,j,M} 2^{k(M-i-j)}|\mathbf x|^{M-j}\leq \sum_i C_{i,j,M} 2^{kM}|\mathbf x|^{M-j}
\end{equation*}
As we explained, taking maximum over $i$, we get the desired \eqref{second part of int by parts}.
\end{proof}

We now turn to tackle the oscillatory integral part. This is just running the integration by parts argument of Stein's book \cite[Proposition 1, Chapter VIII]{bible}. To maintain the information of the support of the functions, we denote
\begin{align*}
  &\int (-2\pi i \mathbf x\cdot \gamma(u))^j e^{-2\pi i 2^k \boldsymbol{\xi}\cdot \gamma(u)}\chi(u)du(\mathfrak{c}(2^k\boldsymbol{\xi})\varphi_0(\boldsymbol{\xi}))\\
  =&\int \psi(u,\boldsymbol{\xi},\mathbf x)e^{i2^k\phi(\boldsymbol{\xi},u)}du\\
  =&\int_{-1.5}^{1.5} \psi(u,\boldsymbol{\xi},\mathbf x)e^{i2^k\phi(\boldsymbol{\xi},u)}du
\end{align*}
The amplitude function is 
\begin{equation*}
    \psi(u,\boldsymbol{\xi},\mathbf x)=(-2\pi i \mathbf x\cdot \gamma(u))^j\chi(u)\mathfrak{c}(2^k\boldsymbol{\xi})\varphi_0(\boldsymbol{\xi})
\end{equation*} 
satisfying that 
\begin{equation}\label{psi}
    |\psi|\leq C_j |\mathbf x|^j.
\end{equation}
 The phase function $\phi(u,\boldsymbol{\xi})$ is 
 \begin{equation*}
     \phi(u,\boldsymbol{\xi})=-2\pi \boldsymbol{\xi}\cdot \gamma(u)=-2\pi (\sum_{i=1}^d\xi_iu^d),
 \end{equation*}
 whose derivative (as a function of $u$) is 
 \begin{equation*}
     \phi'(u,\boldsymbol{\xi})=-2\pi(\boldsymbol{\xi}\cdot \gamma(u))'=-2\pi(\xi_1+2u\xi_2+...+du^{d-1}\xi_d).
 \end{equation*}
 Due to the existence of $\mathfrak{c}$ and $\varphi_0$, it is easy to verify that there is no critical point in the supported of this two functions and  for all $\boldsymbol{\xi}\in$ supp$(\mathfrak{a})\cap $ supp$(\varphi_0)$ and $u\in [-1.5,1.5]$,
 \begin{equation*}
     |\phi'(u,\boldsymbol{\xi})|\geq c_d>0.
 \end{equation*}
 Similar to the argument in the proof of Lemma \ref{intbypar}, after applying an integration by parts argument for $J$ times, we obtain (details can be found in Stein's book)
 \begin{equation}\label{firstpartintbyparts}
     \int_{-1.5}^{1.5} \psi(u,\boldsymbol{\xi},\mathbf x)e^{i2^k\phi(\boldsymbol{\xi},u)}du=\int_{-1.5}^{1.5}  ({}^tD)^J(\psi)\cdot e^{i2^k\phi(\boldsymbol{\xi},u)} du,
 \end{equation}
 for a differential operator with the following form,
 \begin{equation*}
     {}^tD(f)(u):=-\frac{d}{du}\left(\frac{f}{2\pi i2^k\phi'(u,\boldsymbol{\xi})}\right).
 \end{equation*}
We will use induction to prove that 
\begin{lemma}\label{firstpartlemma} For any $J$, there is a constant $C_{j,J}$ such that 
    \begin{equation}\label{first part estimate}
        |({}^tD)^J(\psi)|\leq C_{j,J}2^{-kJ}|\mathbf x|^j.
    \end{equation}
\end{lemma}
Take this for granted and plug \eqref{first part estimate} to \eqref{firstpartintbyparts}, we get 
\begin{equation}\label{firstpartfinalestimate}
   \left| \int_{-1.5}^{1.5} \psi(u,\boldsymbol{\xi},\mathbf x)e^{i2^k\phi(\boldsymbol{\xi},u)}du\right|=\left|\int_{-1.5}^{1.5} ({}^tD)^J(\psi)\cdot e^{i2^k\phi(\boldsymbol{\xi},u)} du\right|\leq C_{j,J} 2^{-kJ}|\mathbf x|^j.
\end{equation}

Plugging \eqref{firstpartfinalestimate} and \eqref{second part of int by parts} back to \eqref{intbyparts2}, we get 
\begin{equation*}
|\text{LHS of \eqref{intbyparts2}}|\leq C_{j,M}2^{-k(J-M)}|\mathbf{x}|^M.
\end{equation*}

Back to the estimate \eqref{dmk}, by choosing the the maximal coefficient, the dependence over $j$ disappears. So we have 
\begin{equation*}
    |(D_{\boldsymbol{\xi}}^\mathbf x)^M(m_k)(\boldsymbol{\xi})|\leq \frac{C_{M,J} }{|\mathbf x|^M}2^{-k(J-M)}.
\end{equation*}
This concludes the proof of Lemma \ref{fast decay for m_k}.
\end{proof}

\begin{proof}[Proof of Lemma \ref{firstpartlemma}]
   When $J=1,$ we have
    \begin{equation}\label{00case}
    \begin{aligned}
        |\psi'&|=\left|\frac{d}{du}({}^tD^0)\psi\right|\\&=|(j(-2\pi i)^j(\mathbf x\cdot \gamma(u))^{j-1}\phi'(\mathbf x)\chi(u)+(-2\pi i \mathbf x\cdot \gamma(u))^j\chi'(u))\mathfrak{c}(2^k\boldsymbol{\xi})\varphi_0(\boldsymbol{\xi})|\\
        &\leq C_j |\mathbf x|^j=C_j 2^{-k0}|\mathbf x|^j
    \end{aligned}
    \end{equation}
    Furthermore, we have 
    \begin{align*}
         |{}^tD(\psi)(u)|&=|\frac{d}{du}\left(\frac{\psi}{2\pi i2^k\phi'(u,\boldsymbol{\xi})}\right)|\\
         &= \frac{C_j}{2^k}|\frac{d}{du}\left(\frac{\psi}{\phi'(u)}\right)|\\
         &=\frac{C_j}{2^k}|\frac{\psi'\phi'(u)-\phi''(u)\psi}{|\phi'(u)|^2}|.
    \end{align*}
    Note that $|\boldsymbol{\xi}|\sim 1,$ so all derivatives of $\phi$ will be bounded (the bounds depend on order of the derivative.) So combining this with \eqref{psi} and \eqref{00case}, we obtain
    \begin{equation*}
         |{}^tD(\psi)(u)|\leq \frac{C|\mathbf x|^j}{2^k}.
    \end{equation*}

    Assume the inequality \eqref{first part estimate} and 
    \begin{equation}\label{induc2}
        \left|\frac{d}{du}({}^tD^{J-1})\psi\right|\leq C_j2^{-k(J-1)}|\mathbf x|^j,
    \end{equation} hold for all $J\leq n-1$. 
    When $J=n$, 
    \begin{equation*}
        \begin{aligned}
             |({}^tD)^n(\psi)(u)|&=\left|\frac{d}{du}\left(\frac{({}^tD)^{n-1}(\psi)(u)}{2\pi i2^k\phi'(u,\boldsymbol{\xi})}\right)\right|\\
         &= \frac{C_j}{2^k}\left|\frac{d}{du}\left(\frac{({}^tD)^n(\psi)(u)}{\phi'(u)}\right)\right|\\
         &=\frac{C_j}{2^k}\left|\frac{({}^tD^{n-1}\psi)'\phi'(u)-\phi''(u)({}^tD^{n-1}\psi)}{|\phi'(u)|^2}\right|.
        \end{aligned}
    \end{equation*}
    For the second term, we can use the induction hypothesis \eqref{first part estimate}, 
    \begin{equation*}
        \frac{C_j}{2^k}\left|\frac{\phi''(u)({}^tD^{n-1}\psi)}{|\phi'(u)|^2}\right|\leq C_{j,n-1}\frac{|\mathbf x|^j}{2^{kn}}.
    \end{equation*}
    For the first term, it suffices to prove 
    \begin{equation}\label{firsthardestimate}
        \left|\frac{d}{du}{}^tD^{n-1}(\psi)\right|\leq C_{j,n}2^{-k(n-1)}|\mathbf x|^{j}.
    \end{equation}
    Calculating the left hand side explicitly, we obtain, 
    \begin{equation*}
        \begin{aligned}
             &\left|\frac{d}{du}{}^tD^{n-1}(\psi)\right|=\left|\frac{d^2}{du^2}\frac{{}^tD^{n-2}(\psi)}{2^k\phi'}\right|=\frac{1}{2^k}\times \\
             &\left|\frac{\frac{d^2}{du^2}{}^tD^{n-2}\psi\cdot \phi-\phi'\frac{d}{du}{}^tD^{n-2}\psi}{|\phi'|^2}-\frac{\frac{d}{du}(\phi''{}^tD^{n-2}\psi)\phi'^2-2\phi'\phi''(\phi'' {}^tD^{n-2}\psi)}{|\phi'|^4}\right|.
        \end{aligned}
    \end{equation*}
    Except for the first term, all remaining terms can be tackled using induction hypothesis \eqref{00case} and \eqref{induc2}. Iterate this process for $O(n)$ steps. It turns out \eqref{firsthardestimate} will be reduced to prove the following  
    \begin{equation}\label{reduction}
        \frac{d^n}{du^n}(\psi)\leq C_{j,n}|\mathbf x|^j.
    \end{equation}
Recall that 
\begin{equation*}
    \psi(u,\boldsymbol{\xi},\mathbf x)=(-2\pi i \mathbf x\cdot \gamma(u))^j\chi(u)\mathfrak{c}(2^k\boldsymbol{\xi})\varphi_0(\boldsymbol{\xi}).
\end{equation*}
    This is another application of induction and  Leibniz's rule similar to the reasoning in the proof of Lemma \ref{secondpartlemma}. So we omit the proof.

    This concludes the proof of the case $J=n$ and we finish the proof of Lemma \ref{firstpartlemma}.
    \end{proof}

    \subsection{Verification of \eqref{symboldecaycondition}}\label{Verification of conditions}
    Recall 
    \[
    a(u,r,\boldsymbol{\xi})=\chi_0(u)\chi_1(r)\phi(\boldsymbol{\xi})\varphi_k(\tilde{\boldsymbol{\xi}}),
\]
and our goal is to verify 
\begin{equation}\label{last}
    |\p^j_u\p^l_r\p^\alpha_{\boldsymbol{\xi}}   a(u,r,\boldsymbol{\xi})|\leq B |\boldsymbol{\xi}|^{-|\alpha|},
\end{equation}
for $(j,l,\alpha)\in\{(j,l,\alpha): 0\leq j\leq 1, 0\leq l\leq 2d, |\alpha|\leq 2d+2\}$.

\begin{proof}
    Since $\chi_i,i=0,1$ are smoothing functions, as long as $B$ is large enough, all of their derivatives with finite order can be bounded by $B$. So it suffices to consider the decay of partial derivative with $\boldsymbol{\xi}.$ 
Note the support of $\phi(\boldsymbol{\xi})\varphi_k(\tilde{\boldsymbol{\xi}})$ is $B(\mathbf 0,O( 2^k))$ It suffices to prove that 
\begin{equation*}
     |\p^\alpha_{\boldsymbol{\xi}}   \phi(\boldsymbol{\xi})\varphi_k(\tilde{\boldsymbol{\xi}})|\leq B 2^{-k|\alpha|}.
\end{equation*}
We take $\alpha=(0,1,...,0)$ for example. For other cases, the computations are similar. For the general case, it is also another Leibniz's rule plus induction argument similar to the proof of Lemma \ref{firstpartlemma} combined with Lemma \ref{secondpartlemma}. So we omit the proof again.
\begin{align*}
    \frac{\p}{\p\xi_2}\phi(\boldsymbol{\xi})\varphi_k(\tilde{\boldsymbol{\xi}})&= \frac{\p}{\p\xi_2}\mathfrak{a}_1\left(\frac{\xi_1}{(100d)^d|\tilde{\boldsymbol{\xi}}|}\right)\varphi_k(\tilde{\boldsymbol{\xi}})\\
    &=\mathfrak{a}_1'\cdot \frac{\p }{\p \xi_2}\left(\frac{\xi_1}{(100d)^d|\tilde{\boldsymbol{\xi}}|}\right)\cdot \varphi_k(\tilde{\boldsymbol{\xi}})+\frac{\p \varphi_k}{\p \xi_2}\mathfrak{a}.
\end{align*}
It is elementary to verify that the absolute value of the second term  is  $\lesssim 2^{-k}$. Indeed, since $|\mathfrak{a}_1|\lesssim 1$, it suffices to consider $\frac{\p \varphi_k}{\p \xi_2}$.  
\begin{equation*}
    \left| \frac{\p \varphi_k}{\p \xi_2}(\tilde{\boldsymbol{\xi}})\right|=  \left| \frac{\p}{\p \xi_2} \varphi(2^{-k}\tilde{\boldsymbol{\xi}})\right|=|2^{-k}\varphi'_{\xi_2}(2^{-k}\tilde{\boldsymbol{\xi}})|\lesssim 2^{-k}.
\end{equation*}
We also have $|\mathfrak{a}'_1|\lesssim 1$, while for 
\begin{equation*}
     \left|\frac{\p }{\p \xi_2}\left(\frac{\xi_1}{(100d)^d|\tilde{\boldsymbol{\xi}}|}\right)\right|=\left|\frac{\xi_1\xi_2}{(100d)^{d}(\sum_{i=2}^d|\xi_i|^2)^{3/2}}\right|
\end{equation*}
since $\max\{|\xi_1,|\xi_2|\}|\leq 2(100d)^d(\sum|\xi_i|^2)^{1/2}\sim 2^{k}$, 
\begin{equation*}
    \left|\frac{\xi_1\xi_2}{(100d)^{d}(\sum_{i=2}^d|\xi_i|^2)^{3/2}}\right|\leq C_d\frac{1}{|\tilde{\boldsymbol{\xi}}|}\leq C_d 2^{-k}.
\end{equation*}
Note that the orders of the partial derivatives depends only on $d$, we can always choose sufficiently large $B$ such that \eqref{last} holds. This concludes the proof.
\end{proof}

\bibliographystyle{plain}
\bibliography{ref-2-2}

\vspace{1em}
\noindent \textsc{Department of Mathematics, 1984 Mathematics Road, The University of British Columbia Vancouver, BC Canada V6T 1Z2}. \\
\textit{Email address}: \texttt{chjwang@math.ubc.ca}

\end{document}